\documentclass[12pt]{article}
\usepackage{epsfig,psfrag,amsmath,amssymb,latexsym}
\usepackage{amscd}
\usepackage{color}
\usepackage{amsfonts}
\usepackage{graphicx}
\usepackage{wasysym}
\usepackage{mathrsfs}
\usepackage{verbatim}
\usepackage{hyperref}
\usepackage[utf8]{inputenc}
\numberwithin{equation}{section}
\newtheorem{thm}{Theorem}[section]

\newtheorem{theorem}[thm]{Theorem}
\newtheorem{fact}[thm]{Fact}
\newtheorem{corollary}[thm]{Corollary}

\newtheorem{definition}[thm]{Definition}
\newtheorem{lemma}[thm]{Lemma}
\newtheorem{remark}[thm]{Remark}
\newenvironment{proof}[1][Proof]{\textbf{#1.} }{\ \rule{0.5em}{0.5em}}

\newcommand{\htwo}{H^{2|2}}
\newcommand{\hthree}{H^{3|2}}

\newcommand{\cut}{{\tt cut}}

\newcommand{\arcosh}{\operatorname{arcosh}}

\newcommand{\artanh}{\operatorname{artanh}}

\newcommand{\body}{\operatorname{body}}

\newcommand{\sk}[1]{\left\langle{#1}\right\rangle}
\newcommand{\cL}{\mathcal{L}}

\newcommand{\cA}{\mathcal{A}}
\newcommand{\cD}{\mathcal{D}}

\newcommand{\cH}{\mathcal{H}}

\newcommand{\B}{\mathcal{B}}

\newcommand{\cE}{{\mathcal E}} 
\newcommand{\cN}{{\mathcal N}} 
  
\newcommand{\R}{{\mathbb R}}  
  
\newcommand{\N}{{\mathbb N}}  
\newcommand{\Z}{{\mathbb Z}}

\newcommand{\T}{{\mathcal{T}}}

\makeatletter
\def\moverlay{\mathpalette\mov@rlay}
\def\mov@rlay#1#2{\leavevmode\vtop{%
   \baselineskip\z@skip \lineskiplimit-\maxdimen
   \ialign{\hfil$\m@th#1##$\hfil\cr#2\crcr}}}
\newcommand{\charfusion}[3][\mathord]{
    #1{\ifx#1\mathop\vphantom{#2}\fi
        \mathpalette\mov@rlay{#2\cr#3}
      }
    \ifx#1\mathop\expandafter\displaylimits\fi}
\makeatother

\begin{document}
\thispagestyle{empty}

\begin{center}
  {\LARGE The non-linear supersymmetric hyperbolic sigma model on a complete graph with hierarchical interactions}\\[3mm]
{\large Margherita Disertori\footnote{Institute for Applied Mathematics
\& Hausdorff Center for Mathematics, 
University of Bonn, \\
Endenicher Allee 60,
D-53115 Bonn, Germany.
E-mail: disertori@iam.uni-bonn.de}
\hspace{1cm} 
Franz Merkl \footnote{Mathematisches Institut, Ludwig-Maximilians-Universit{\"a}t  M{\"u}nchen,
Theresienstr.\ 39,
D-80333 Munich,
Germany.
E-mail: merkl@math.lmu.de
}
\hspace{1cm} 
Silke W.W.\ Rolles\footnote{
Zentrum Mathematik, Bereich M5,
Technische Universit{\"{a}}t M{\"{u}}nchen,
Boltzmannstr.\ 3,
D-85748 Garching bei M{\"{u}}nchen,
Germany.
E-mail: srolles@ma.tum.de}
\\[3mm]
{\small \today}}\\[3mm]
\end{center}

\begin{abstract}
We study the non-linear supersymmetric hyperbolic sigma model $\htwo$
on a complete graph with hierarchical interactions. 
For interactions which do not decrease too fast in the hierarchical 
distance, we prove tightness
of certain spin variables in horospherical coordinates, uniformly
in the pinning and in the size of the graph. The proof relies on 
a reduction to an effective $\htwo$ model; its size is logarithmic 
in the size of the original model.
\end{abstract}


\section{Introduction}

\subsection{History of the $\htwo$-model and related models}
The supersymmetric hyperbolic sigma model, $\htwo$-model for short, 
was introduced by Zirnbauer in \cite{zirnbauer-91} as a 
``toy model for studying diffusion and localization in disordered 
one-electron systems''. It is a supersymmetric variant of a statistical 
mechanics model on a finite graph. The model deals with spin variables 
taking values in the 
supermanifold $\htwo$ with two Grassmann components over the hyperboloid 
$H^2=\{(x,y,z)\in\R^3:x^2+y^2-z^2=-1, z>0\}$. The model has supersymmetries, which 
extend the generators of the Lorentz group acting on $H^2$. In 
\cite{disertori-spencer-zirnbauer2010},
Disertori, Spencer, and Zirnbauer examine this model over boxes in $\Z^d$, 
$d\ge 3$. For sufficiently small temperature, they derive moment estimates
and conclude that the spins are aligned with high probability. 
For high temperature in any dimension $d$, Disertori and Spencer 
\cite{disertori-spencer2010} show exponential decay of the $y$-$y$-two-point function 
for the $\htwo$-model. The result holds for any temperature in dimension one.
In two dimensions, polynomial decay of the correlations was shown for all 
temperatures independently by Kozma and Peled \cite{kozma-peled2021} and Sabot
\cite{sabot-polynomial-decay2021}, using different tools. 

Surprisingly, although the model was originally designed to describe 
condensed matter systems, there is a close connection with vertex-reinforced 
jump processes as was observed by Sabot and Tarr\`es \cite{sabot-tarres2012}. 
All spin variables of the $\htwo$-model written in horospherical coordinates
have a probabilistic interpretation in terms of the vertex-reinforced jump
process as is worked out in \cite{merkl-rolles-tarres2019}. 
 
There is another surprising connection of the $\htwo$-model to certain random 
Schrödinger operators discovered by Sabot, Tarr\`es, and Zeng in 
\cite{sabot-tarres-zeng2017} and \cite{sabot-zeng15}. For the analysis of 
these random Schrödinger operators, certain martingales derived from the
spin variables of the $\htwo$-model play an essential role. These 
martingales are generalized in \cite{disertori-merkl-rolles2017}
and supersymmetric variants of them are described in 
\cite{disertori-merkl-rolles2019}. 
 
The classical isomorphism theorems relating the local times of a continuous-time 
random walk to the square of a Gaussian free field are generalized by 
Bauerschmidt, Helmuth, and Swan in \cite{bauerschmidt-helmuth-swan-dynkin-isomorphism2021}
to spin systems taking values in hyperbolic and spherical geometries.

There are generalizations of the $\htwo$-model, called $H^{p|2q}$-models, 
which deal with $p$ real-valued spin variables and $2q$ Grassmann variables per vertex. 
For $p=2$, these models are examined first by Crawford in \cite{crawford2021}. 
The case of the $H^{0|2}$-model, 
which describes the arboreal gas (configurations of unrooted spanning forests),
and the closely related $H^{2|4}$-model
are examined by Bauerschmidt, Crawford, Helmuth, and Swan in 
\cite{bauerschmidt-crawford-helmuth-swan-spanning-forests2021} and 
\cite{bauerschmidt-crawford-helmuth-percolation-transition2021}.

In particular, the $H^{0|2}$-model can be represented by a purely Fermionic
model with a Gaussian measure perturbed by a short range interaction. This
allows to apply a rigorous renormalization group analysis in 
\cite{bauerschmidt-crawford-helmuth-percolation-transition2021}.
It is not clear how to extend such an analysis to the $\htwo$-model, the 
main obstruction being the hyperbolic structure and the presence
of long range interactions. If one considers the corresponding 
vertex-reinforced jump process on hierarchical lattices, informal computations 
by Abdesselam, Brydges, Helmuth, Lupu, Sabot, Swan, and Zeng
done during an AIM workshop\footnote{AIM Workshop on 
Self-interacting processes, supersymmetry, and Bayesian statistics 
organized by P.\ Diaconis, M.\ Disertori, C.\ Sabot, and P.\ Tarr\`es, 2019,\\
https://aimath.org/pastworkshops/selfsuperbayes.html} 
suggest, for a special choice of the pinning and the interactions, the existence of a 
renormalization semigroup operation, which leaves the 
process invariant under appropriate scaling. 

A recent survey on the subject is given by Bauerschmidt and Helmuth in 
\cite{bauerschmidt-helmuth-survey2021}.

\subsection{The $H^{2|2}$-model on general graphs}

\paragraph{Overview of the paper.}
In this paper, we deal with the $\htwo$-model on a complete graph of size $2^N$
with interactions depending on the hierarchical distance in a binary tree. 
For some observables this model can be reduced to an effective $\htwo$-model
with only $O(N)$ vertices. This reduction is described in Theorem 
\ref{thm:reduced-model} below. As an application of this reduction, we prove
tightness of differences of certain spin variables in the original model, 
provided that the interactions do not decay too fast with the hierarchical
distance. This is made precise in Theorem \ref{thm:main}. 

\paragraph{Introduction of the model.}
Let $\Lambda$ be a finite non-empty set. We view it as the vertex set of a
complete undirected graph with edge set $\cE_\Lambda:=\{\{i,j\}:i,j\in\Lambda\}$. 
To construct the supermanifold $\htwo$, we introduce a Grassmann algebra 
$\cA=\cA_0\oplus\cA_1$ with the even 
(commuting) subalgebra $\cA_0$ and the odd (anticommuting) subspace 
$\cA_1$; see \cite{phdthesis-swan2020} for more details. To keep the 
exposition self-contained, we also give some details in the appendix. 
We assume that $\cA$ is large enough 
so that we can associate to each vertex $i\in\Lambda$ three variables 
$x_i,y_i,z_i$ taking values in $\cA_0$ and linearly 
independently two odd (Grassmann) variables $\xi_i,\eta_i\in\cA_1$.
Each even element $x\in\cA_0$ can be decomposed in a unique way in a 
real part $\body(x)\in\R$ and a nilpotent part $x-\body(x)$. 
We define $\hthree:= \cA_0^3\times\cA_1^2$ and endow it with the inner product 
\begin{align}
\label{eq:def-inner-product}
\sk{v,v'}:=xx'+yy'-zz'+\xi\eta'-\eta\xi'
\end{align}
for $v=(x,y,z,\xi,\eta),v'=(x',y',z',\xi',\eta')\in\hthree$. 
The supermanifold $\htwo$ can now be defined by 
\begin{align}
\htwo=&\{v=(x,y,z,\xi,\eta)\in\hthree:\,\sk{v,v}=-1,\body(z)>0\}\nonumber\\
=&\{(x,y,z,\xi,\eta)\in\hthree:\, x^2+y^2-z^2+2\xi\eta=-1,\body(z)>0\}.
\end{align}

For any smooth function $f:\R^k\to\R$ the corresponding
extension to a superfunction $f:\cA_0^k\to\cA_0$ is denoted by the same
symbol $f$ and is constructed by a Taylor expansion in the nilpotent parts. 

The superintegration form $\cD v$ on $H^{2|2}$ is 
defined by 
\begin{align}
f\mapsto\int_{H^{2|2}}\cD v f(v):=
\frac{1}{2\pi}\int_{\R^2}dx\, dy\,
\partial_\xi\partial_\eta\left(\frac{1}{z}f(x,y,z,\xi,\eta)\right),
\end{align}
where
\begin{align}
  z=\sqrt{1+x^2+y^2+2\xi\eta}
=\sqrt{1+x^2+y^2}+\frac{\xi\eta}{\sqrt{1+x^2+y^2}}
\end{align}
according to the constraint $v\in H^{2|2}$. Here, $f$ is any superfunction 
decaying sufficiently fast to make the integral well-defined. 

We abbreviate 
\begin{align}
\label{eq:def-o}
o:=(0,0,1,0,0)\in H^{2|2}. 
\end{align}
We take a weight function $W:\cE_\Lambda\to[0,\infty)$, $\{i,j\}\mapsto W_{ij}=W_{ji}$ with $W_{ii}=0$ for all $i\in\Lambda$ 
and consider the set of edges 
\begin{align}
\label{eq:E-Lambda-W}
\cE_{\Lambda,W}:=\{\{i,j\}\in\cE_\Lambda:W_{ij}>0\}.
\end{align}
Furthermore, we take a pinning strength function $h:\Lambda\to[0,\infty)$, 
$i\mapsto h_i$. Throughout the paper, we assume that it fulfills 
$h_i>0$ for at least one $i$ in every connected component of the graph 
$(\Lambda,\cE_{\Lambda,W})$. This connectivity assumption guarantees that 
the following $\htwo$-model is well-defined:
its action with couplings $W$ and pinning $h$ 
is defined for $v_\Lambda:=(v_j)_{j\in\Lambda}\in(\htwo)^\Lambda$ by 
\begin{align}
S_{W,h}^\Lambda(v_\Lambda):=
&-\frac12\sum_{i,j\in\Lambda}W_{ij}\sk{v_i-v_j,v_i-v_j}
  +\sum_{i\in\Lambda}h_i(1+\sk{v_i,o})\nonumber\\
=&\sum_{i,j\in\Lambda}W_{ij}(1+\sk{v_i,v_j})
  +\sum_{i\in\Lambda}h_i(1+\sk{v_i,o}).
\label{eq:action-S}
\end{align}
The $H^{2|2}$-model
over $\Lambda$ is given by the superintegration form 
\begin{align}
f\mapsto E_{W,h}^\Lambda[f(v_\Lambda)]:=\int_{(H^{2|2})^\Lambda}\cD v_\Lambda
\left( e^{S^\Lambda_{W,h}(v_\Lambda)} f(v_\Lambda)\right)
\label{eq:def-superintegration-form}
\end{align}
with the abbreviation
\begin{align}
\label{eq:Dv-Lambda}
\cD v_\Lambda:= \prod_{j\in\Lambda}\cD v_j.
\end{align}
By supersymmetry, $E^\Lambda_{W,h}[1]=1$ for the constant function 
$f(v_\Lambda)\equiv 1$, see formula (5.1) in \cite{disertori-spencer-zirnbauer2010}. 

\paragraph{Horospherical coordinates.}
For any vertex $i\in\Lambda$ and $v_i=(x_i,y_i,z_i,\xi_i,\eta_i)\in\htwo$, 
we introduce horospherical coordinates consisting of even components 
$u_i,s_i$ and Grassmann variables $\overline\psi_i$, $\psi_i$ as follows, 
cf.\ Section 2.2 in \cite{disertori-spencer-zirnbauer2010}: 
\begin{align}
u_i=\log (x_i+z_i),\quad 
 s_i=\frac{y_i}{x_i+z_i},\quad
\overline\psi_i=\frac{\xi_i}{x_i+z_i},\quad  
\psi_i=\frac{\eta_i}{x_i+z_i}. 
\end{align}
The inverse transformation is given by 
\begin{align}
\label{eq:change-of-coordinates1}
& x_i=\sinh u_i-\left(\frac12s_i^2+\overline\psi_i\psi_i\right)e^{u_i}, \quad
y_i=s_ie^{u_i}, \quad
\xi_i=\overline\psi_i e^{u_i}, \quad
\eta_i=\psi_i e^{u_i}, \\
\label{eq:change-of-coordinates2}
& z_i=\sqrt{1+x^2_i+y^2_i+2\xi_i\eta_i}=\cosh u_i+\left(\frac12s_i^2+\overline\psi_i\psi_i\right)e^{u_i}.
\end{align}
In these coordinates, the even components 
$(u_\Lambda,s_\Lambda)$ restricted to $\R^\Lambda\times\R^\Lambda$ have 
the following probabilistic interpretation: there is a unique 
probability measure $P^\Lambda=P_{W,h}^\Lambda$ on $(\R^\Lambda)^2$ 
such that for any smooth bounded test function $f:(\R^\Lambda)^2\to\R$
one has 
\begin{align}
  \int_{(\R^\Lambda)^2}f\, dP_{W,h}^\Lambda
=E_{W,h}^\Lambda[f(u_\Lambda(v_\Lambda),s_\Lambda(v_\Lambda))].
\end{align}
Using the abbreviations
\begin{align}
\label{eq:def-Bij}
& B_{ij}:=\cosh(u_i-u_j)+\frac12(s_i-s_j)^2e^{u_i+u_j}, \quad
B_j:=\cosh u_j+\frac12 s_j^2e^{u_j}, \\
& M_{ij}:=1_{\{i=j\}}\Big(h_ie^{u_i}+\sum_{k\in\Lambda}W_{ik}e^{u_i+u_k}\Big)
-1_{\{i\neq j\}}W_{ij}e^{u_i+u_j}
\end{align}
for $i,j\in\Lambda$, the corresponding probability density equals
\begin{align}
\label{eq:P-in-horosph-coor}
  P_{W,h}^\Lambda(du_\Lambda\, ds_\Lambda)
=e^{-\frac12\sum_{i,j\in\Lambda}W_{ij}(B_{ij}-1)}
e^{-\sum_{i\in\Lambda}h_i(B_i-1)}\det (M_{ij})_{i,j\in\Lambda}
\prod_{i\in\Lambda}\frac{e^{-u_i}\, du_i\, ds_i}{2\pi} ,
\end{align}
where $du_i\, ds_i$ is the Lebesgue measure on $\R^2$, cf.\ 
\cite{disertori-spencer-zirnbauer2010}. 
In the rest of the paper we mostly work with the $v$-variables.

\section{Hierarchical $\htwo$-model and results} 
\subsection{Definition and tightness}
\label{subsec:hierarchical-h22}
We now consider the special case
that $\Lambda$ consists of $2^N$ vertices with some given $N\in\N_0$, where
every vertex $i$ interacts with every other vertex $j$ in a hierarchical manner.
The hierarchical structure is best visualized by the binary tree with set of leaves
equal to the vertex set $\Lambda$. 
More precisely, for any $n\in\N_0$, let $\Lambda_n=\{0,1\}^n$ be the set of binary words of 
length $n$. In particular, $\Lambda_0=\{\emptyset\}$ and $\Lambda=\Lambda_N$. The reader may 
visualize the words as binary numbers $0,\ldots,2^N-1$ with reversed order
of digits.
Let
$\T^N:=\bigcup_{n=0}^N\Lambda_n$
be the vertex set of a binary tree of height $N$ with root $\emptyset$. 
For $j\in\Lambda_n$, let $|j|=n$ be the length of the word $j$, which is 
the distance of $j$ from the root.
The elements of $\Lambda_N$ are the leaves of the tree. 
An illustration of the tree $\T^3$ is given in Figure \ref{fig1}. 

\begin{figure}
\centering
\includegraphics[width=0.8\textwidth]{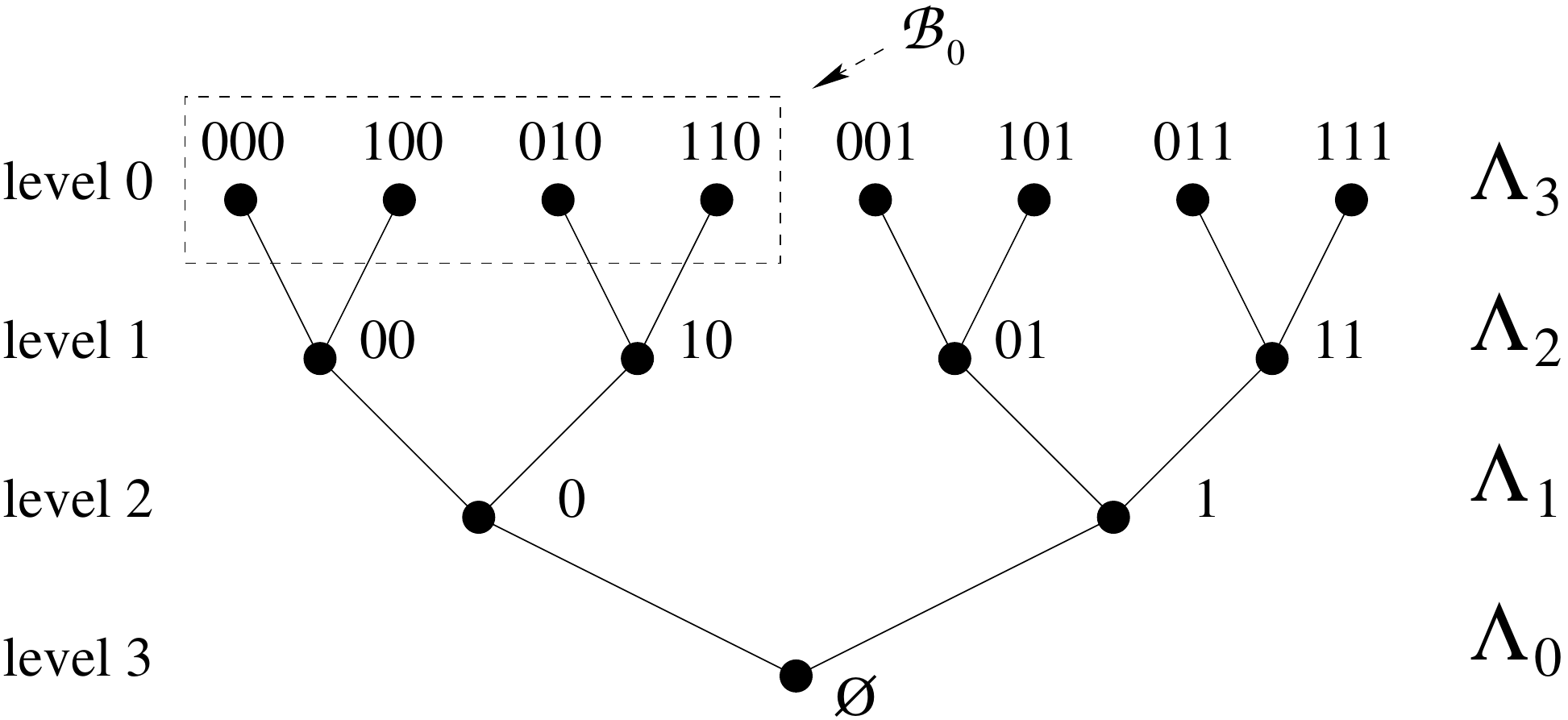}
\caption{The tree $\T^3=\Lambda_3\cup\Lambda_2\cup\Lambda_1\cup\Lambda_0$.}
\label{fig1}
\end{figure}

For $j=(j_0,\ldots,j_{|j|-1})\in\T^N\setminus\{\emptyset\}$, let 
$\cut(j):=(j_1,\ldots,j_{|j|-1})$ denote the vertex directly below $j$. 
The \emph{hierarchical distance} between $i,j\in\Lambda_N$ is defined as
\begin{align}
d(i,j):=\min\{l\in\N_0:\cut^l(i)=\cut^l(j)\}.
\end{align}
Here, $\cut^l$ denotes the $l$-fold iteration of $\cut$ and $\cut^0$ is the 
identity map. The hierarchical distance is the distance to the first point 
in $\T^N$ where the shortest paths from $i$ and $j$ to the root meet. 
It is also half the graph distance between $i$ and $j$ in the tree graph
corresponding to $\T^N$. The \emph{minimum} of $i,j\in\Lambda_N$ is defined by
\begin{align}
i\wedge j:=\cut^{d(i,j)}(i)=\cut^{d(i,j)}(j).   
\end{align}
More generally, for $i,j\in\T^N$, we define $i\wedge j$ to be the least
common ancestor of $i$ and $j$, which equals $\cut^m(i)$ with the smallest
$m\in\N_0$ such that $\cut^m(i)=\cut^{m'}(j)$ for some $m'\in\N_0$. 

Take a weight function $w:\N_0\to[0,\infty)$. It encodes
the weight function $W:\cE_{\Lambda_N}\to[0,\infty)$ on the set of edges $\cE_{\Lambda_N}$
as follows: for $i,j\in\Lambda_N$, 
\begin{align}
\label{eq:def-W-hierarchical}
W_{ij}:=w(d(i,j)). 
\end{align}
Thus, it depends only on the hierarchical distance of $i$ and $j$. In this paper
we study the $\htwo$-model on $\Lambda=\Lambda_N$ with hierarchical weight function 
$W$ given
by \eqref{eq:def-W-hierarchical} and a pinning function $h=(h_i)_{i\in\Lambda_N}$. 
According to formula \eqref{eq:action-S}, the action is given by
\begin{align}
S_{W,h}^{\Lambda_N}(v_{\Lambda_N})
=&\sum_{i,j\in\Lambda_N}w(d(i,j))(1+\sk{v_i,v_j})
  +\sum_{i\in\Lambda_N}h_i(1+\sk{v_i,o}).
\label{eq:action-S-hierarchical}
\end{align}
Important particular cases are constant pinning $h=\epsilon 1=(\epsilon)_{i\in\Lambda_N}$ 
and pinning at one point ${i_0}\in\Lambda_N$, i.e.\ 
$h=\epsilon\delta_{i_0}=(\epsilon 1_{\{i={i_0}\}})_{i\in\Lambda_N}$, with some 
pinning strength $\epsilon>0$. We show that for this particular choice of the
pinning the horospherical 
coordinates $(u_i)_{i\in\Lambda_N}$ are tight with respect to $P^{\Lambda_N}_{W,h}$, 
uniformly in $N$ and the pinning strength $\epsilon>0$ if the weight function 
$w$ does not decrease too fast.

\begin{theorem}[Tightness]
\label{thm:main}
Let $w:\N_0\to(0,\infty)$ and $W:\cE_{\Lambda_N}\to(0,\infty)$ be weight functions
related by \eqref{eq:def-W-hierarchical} such that 
the numbers $\beta_l:=2^{2l+1}w(l+2)$, $l\in\N_0$, fulfill 
\begin{align}
\label{eq:assumption-main-thm}
  \sum_{l=0}^\infty\sqrt{\frac{\log\max\{\beta_l,e\}}{\beta_l}}<\infty.
\end{align}
Then one has 
\begin{align}
\label{eq:claim-main-thm}
\lim_{M\to\infty}\sup_{N\in\N}\sup_h\sup_{p,q\in\Lambda_N}P^{\Lambda_N}_{W,h}(|u_p-u_q|\ge M)
=0,
\end{align}
where the supremum over $h$ is taken over 
all uniform pinning functions $h=\epsilon 1$ and all pinning functions at one point
$h=\epsilon\delta_{i_0}$ with $\epsilon>0$ and ${i_0}\in\Lambda_N$. 
\end{theorem}

A key tool for this result is a representation of expectations 
in terms of an effective reduced model.

\subsection{Effective model} \label{section2.2}
It turns out that 
the value $E^{\Lambda_N}_{W,h}[e^{\sum_{j\in\Lambda}\sk{a_j,v_j}}]$ of the superintegration form 
agrees, for any suitable $a_j$ and $h,$ with the corresponding value of the superintegration 
form for an $\htwo$-model with a smaller vertex set, rescaled weights 
and rescaled pinning. While the original model has $|\Lambda_N|=2^N$
vertices, the effective model used in the proof of Theorem \ref{thm:main}
has only $O(N)$ vertices, resulting in an exponential reduction of 
complexity; see Remark \ref{rem:exp-complexity} for more details. 
To define this effective model we need to introduce some notation.

\paragraph{Block spin variables. }
For $i,j\in\T^N$, we write $i\preceq j$ if $i=\cut^l(j)$ for some
$l\in\N_0$, which means that $j$ is an extension of $i$.  
Let $\B_i:=\{j\in\Lambda_N:i\preceq j\}$ denote the set of leaves above $i$
in $\Lambda_N$. In particular, $\B_\emptyset=\Lambda_N$. 
For any vertex $j\in\T^N$, let 
\begin{align}
\ell(j):=N-|j|  
\end{align}
denote its level. 
All vertices in $\Lambda_N$ are at level 0 and the root $\emptyset$ 
is at level $N$; see Figure \ref{fig1} for an illustration. 
Note that for 
$i\in\Lambda_n$, one has $|\B_i|=2^{N-n}=2^{\ell(i)}$ and for $i,j\in \Lambda_{N}$ we have
$\ell(i\wedge j)=d (i,j)$. 
Starting with the variables $v_j$ with $j\in\Lambda_N$ at level $0$, we 
attach \emph{block spin variables} to all vertices of the binary tree $\T^N$ 
as follows. For $j\in\T^N$, let
\begin{align}
v_j=(x_j,y_j,z_j,\xi_j,\eta_j)
:=\frac{1}{|\B_j|}\sum_{i\in \B_j}v_i
=\frac{1}{2^{\ell(j)}}\sum_{i\in \B_j}v_i.
\label{eq:def-block-spin}
\end{align}
For $j\in\Lambda_N$, this redefinition preserves the original meaning of $v_j$. 

\paragraph{Antichains.}
For the effective model, we choose the vertex set $A\subseteq\T^N$ 
to be an antichain, meaning that for all $i,j\in A$ with 
$i\neq j$ one has $i\not\preceq j$ and $j\not\preceq i$. An antichain 
$A$ is called maximal if for any antichain $A'\supseteq A$ one has $A'=A$. 
Equivalently, $A\subseteq\T^N$ is a maximal antichain if and only if for 
every leaf $i\in\Lambda_N$ there is precisely one $j\in A$ with 
$j\preceq i$. For a maximal antichain $A$, the set 
$A^\uparrow:=\{i\in\T^N:j\preceq i\text{ for some }j\in A\}$ consists of 
all vertices above $A$. 
Given two maximal antichains $A$ and $A'$, we write $A\preceq A'$ if 
$A'\subseteq A^\uparrow$. The notation $A\prec A'$ means that 
$A\preceq A'$ and $A\neq A'$. 

\paragraph{Rescaled weights and pinning.}
We extend the weight function $W:\cE_{\Lambda_N}\to[0,\infty)$ to
$W:\{\{i,j\}:i,j\in\T^N\}\to[0,\infty)$ by
\begin{align}
  W_{ij}:=2^{\ell(i)+\ell(j)}w(\ell(i\wedge j)).
\label{eq:extension-def-W}
\end{align}
Note that, for $i,j\in \Lambda_{N}$ this definition coincides with the original weight
function $W_{ij}= w(d (i,j)).$

Assume now $h:\Lambda_N\to[0,\infty)$ is a pinning function such that $h_{j}>0$ for at least one point
in each connected component of the graph 
$(\Lambda_{N},\cE_{\Lambda_{N},W})$  (cf. remarks below  \eqref{eq:E-Lambda-W}).
We call a maximal antichain $A$ \emph{compatible with $h$},
if for all $j\in A$ the restriction 
$h|_{\B_j}$ of $h$ to $\B_j$ is constant; see Figure \ref{fig:antichain2}
for an illustration. 

Given a maximal antichain $A$ compatible with a pinning function 
$h$, we define an extended pinning function $H:A^\uparrow\to[0,\infty)$ by setting
\begin{align}
  H_j=\sum_{i\in \B_j}h_i\text{ for }j\in A^\uparrow.
\label{eq:def-H}
\end{align}
Note that all summands in the last sum are equal and hence 
$H_j=2^{\ell(j)}h_i$ for any $i\in \B_j$, $j\in A^\uparrow$.

For any $\Lambda\subseteq A $ the notation 
$S^\Lambda_{W,H}$, $E^\Lambda_{W,H}$ 
means  the weight function $W$ is restricted to the underlying edge set $\cE_\Lambda$ 
and the pinning $H$ is restricted to the vertex set $\Lambda$. 

\paragraph{Reduction of the model.}
We now compare the hierarchical $\htwo$-model over the set $\Lambda_N$ 
of leaves of the tree $\T^N$ with arbitrary pinning $h$ 
with the $\htwo$-model over a maximal antichain $A\subseteq\T^N$ compatible
with $h$. 

Recall the even and odd Grassmann subspaces $\cA_0$ and $\cA_1$, 
respectively, and $\hthree=\cA_0^3\times\cA_1^2$
endowed with the inner product $\sk{\cdot,\cdot}$ defined in
equation \eqref{eq:def-inner-product}. 
We extend the comparison relations $<, \le, >, \ge$ on $\R$ to even elements 
$a,b\in\cA_0$ by defining $a<b$ to mean $\body(a)<\body(b)$, and similarly 
for $\le$, $>$, and $\ge$. 
We abbreviate 
\begin{align}
\hthree_+:=&\{v=(x,y,z,\xi,\eta)\in\hthree:\sk{v,v}<0,z>0\}, \\
\hthree_{+0}:=&\{v=(x,y,z,\xi,\eta)\in\hthree:\sk{v,v}\le 0,z\ge 0\}. 
\end{align}

For the next theorem and its corollary, 
consider the extended weight function $W$ from \eqref{eq:extension-def-W}, 
a pinning function $h:\Lambda_N\to[0,\infty)$, a maximal 
antichain $A\subseteq\T^N$ being compatible with $h$, and the extension 
$H$ of $h$ given in \eqref{eq:def-H}. 
In analogy to \eqref{eq:E-Lambda-W}, we consider the graph $(A,\cE_{A,W})$ 
with edge set $\cE_{A,W}:=\{\{i,j\}: i,j\in A,W_{ij}>0\}$. 

\begin{figure}
\centering
\includegraphics[width=0.8\textwidth]{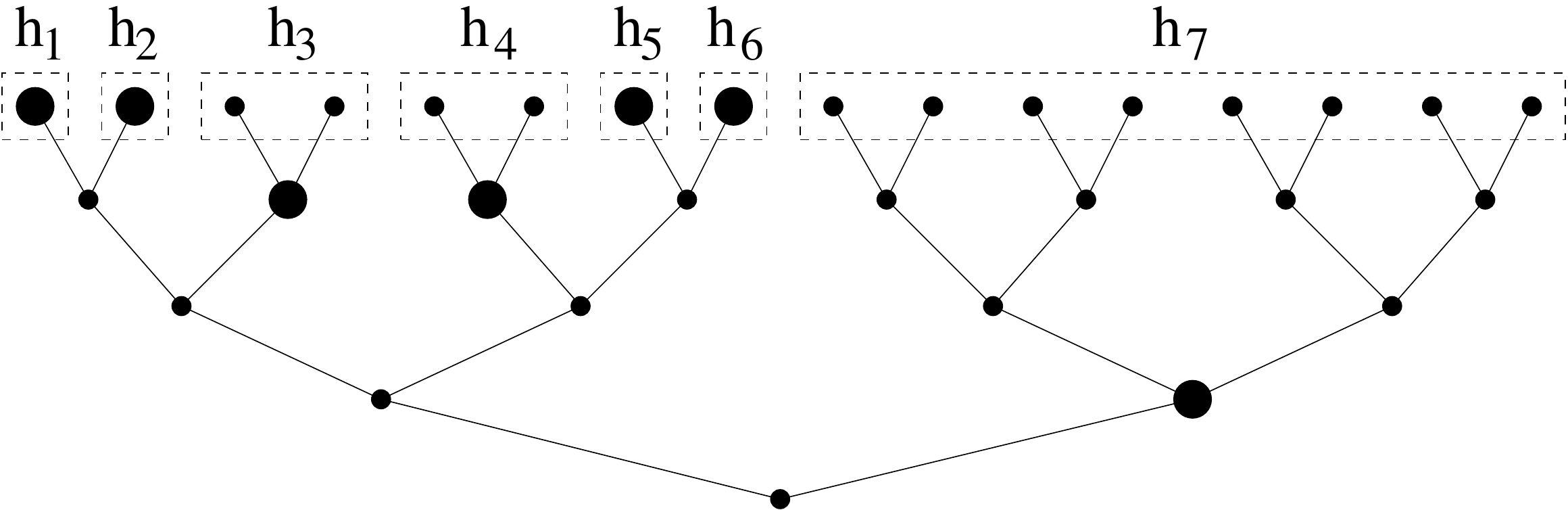}
\caption{The antichain consisting of the big dots is compatible with pinning 
functions given by $h_i$, $1\le i\le 7$.}
\label{fig:antichain2}
\end{figure}

\begin{theorem}[Reduction to the effective model]
\label{thm:reduced-model}
For all $(a_j)_{j\in A}\in(\hthree)^A$ such that $H_jo+a_j\in\hthree_{+0}$ for 
all $j\in A$ and $H_jo+a_j\in\hthree_+$ for at least one $j$
in every connected component of the graph $(A,\cE_{A,W})$, one has 
\begin{align}
E^{\Lambda_N}_{W,h}[e^{\sum_{j\in A}\sk{a_j,v_j}}]=E^A_{W,H}[e^{\sum_{j\in A}\sk{a_j,v_j}}].
\label{eq:expectation-reduced-model}
\end{align}
\end{theorem}
Note that in formula \eqref{eq:expectation-reduced-model} the variables $v_j$ 
have different meanings on the left-hand side and on the right-hand side. 
More precisely, on the left-hand side, they are understood as block spin
variables as in formula \eqref{eq:def-block-spin}, while on the right-hand side
they are the original variables $v_j\in\htwo$ of the $\htwo$-model with 
vertex set $A$.

We remark that taking derivatives with respect to the components of 
the parameters $a_j$ in formula \eqref{eq:expectation-reduced-model}, 
one can treat also polynomials in the $x_j,y_j,z_j,\xi_j,\eta_j$. 
We expect that this can be extended to general observables 
$f((v_j)_{j\in A})$ rather than only $e^{\sum_{j\in A}\sk{a_j,v_j}}$ by 
working out a theory for super Fourier Laplace transforms
of super measures in the spirit of \cite{fresta2021}. 
However, for our application, we need only 
a very special case which is treated in the following corollary in 
an elementary way.

Let $\cL_P((X_i)_{i\in I})$ denote the joint law of the random variables $(X_i)_{i\in I}$ 
with respect to some probability law $P$. Recall the assumptions specified right above
Theorem \ref{thm:reduced-model}. 

\begin{corollary}[Coincidence of distributions]
\label{cor:distribution-effective-model}
Let $B:=A\cap\Lambda_N$ denote the set of leaves of the maximal antichain $A$. 
Then, the joint laws of $u_B:=(u_j)_{j\in B}$ and 
$s_B:=(s_j)_{j\in B}$ with respect to 
$P^{\Lambda_N}_{W,h}$ and $P^A_{W,H}$ coincide. More generally, 
the following joint laws agree
\begin{align}
\label{eq:identity-laplace-trafos}
\cL_{P^{\Lambda_N}_{W,h}}
\Big(\Big(|\B_j|^{-1}\sum_{i\in \B_j}e^{u_i},
|\B_j|^{-1}\sum_{i\in \B_j}s_ie^{u_i}\Big)_{j\in A}\Big)  
=\cL_{P^A_{W,H}}((e^{u_j},s_je^{u_j})_{j\in A}).  
\end{align}
\end{corollary}

\section{Derivation of the effective model}
\subsection{Reduction of certain integrals}

The (super-)symmetries of the $\htwo$-model allow to calculate certain 
integrals. Lemma \ref{le:integrate-variables-out} below provides such a 
result. It is crucially used in the proof of Theorem \ref{thm:reduced-model}. 

\begin{definition}[Fast decaying superfunctions]\mbox{}\\
For any finite set $\Lambda$, a function 
$f:(\hthree)^\Lambda\to\cA_0$, possibly depending on parameters, is called
fast decaying on $(\htwo)^\Lambda$ if it is a smooth superfunction and 
all coefficients of 
$(\R^2)^\Lambda\ni(x_j,y_j)_{j\in\Lambda}\mapsto f((x_j,y_j,\sqrt{1+x_j^2+y_j^2+2\xi_j\eta_j},\xi_j,\eta_j)_{j\in\Lambda})$ in the sense of formula 
\eqref{eq:representation-omega} in the appendix are Schwartz functions.
\end{definition}

For $a\in\hthree_{+}$, we set $\|a\|:=\sqrt{-\sk{a,a}}$. Note that for $a\in\hthree_{+0}$ with
$\body \sk{a,a}=0$ and non zero Grassmann components the norm cannot be defined.
Recall the abbreviation $\cD v_\Lambda$ introduced in formula 
\eqref{eq:Dv-Lambda}. 

\begin{lemma}[Reduction to single spins]
\label{le:integrate-variables-out}
For any finite set $\Lambda$, $c_j\in[0,\infty)$, $j\in\Lambda$, 
$a\in\hthree_+$, and any function $f:(\hthree_+)^\Lambda\to\cA_0$ such that 
$v_\Lambda\mapsto f\left((\sk{v_i,v_j})_{i,j\in\Lambda}\right)e^{\sum_{j\in\Lambda}c_j\sk{v_j,a}}$
is fast decaying on $(\htwo)^\Lambda$, 
one has 
\begin{align} 
&\int_{(H^{2|2})^\Lambda}\cD v_\Lambda \, 
f\left((\sk{v_i,v_j})_{i,j\in\Lambda}\right)e^{\sum_{j\in\Lambda}c_j\sk{v_j,a}}\nonumber\\
=&f\left((-1)_{i,j\in\Lambda}\right)e^{-\sum_{j\in\Lambda}c_j\|a\|}
=\int_{H^{2|2}}\cD v \, 
f\left((\sk{v,v})_{i,j\in\Lambda}\right)e^{\sum_{j\in\Lambda}c_j\sk{v,a}}.
\label{eq:localization-applied}
\end{align}
\end{lemma}
This lemma has the following simple but interesting consequence. 

\begin{corollary}[Super Laplace transform of a block spin]
\label{cor:expectation-mean}
Let $\epsilon>0$ and consider the $\htwo$-model on an arbitrary finite set 
$\Lambda$ with arbitrary weights $W$ and uniform
pinning given by $h_j=|\Lambda|^{-1}\epsilon$ for all $j\in\Lambda$. For all 
$a\in\hthree$ with $a+\epsilon o\in\hthree_+$, the super Laplace 
transform of $|\Lambda|^{-1}\sum_{j\in\Lambda}v_j$ is given by 
\begin{align}
E^\Lambda_{W,h}\left[e^{\sk{a,|\Lambda|^{-1}\sum_{j\in\Lambda}v_j}}\right]
=e^{\epsilon -\|a+\epsilon o\|}=e^{\epsilon-\sqrt{\epsilon^2-2\epsilon\sk{a,o}-\sk{a,a}}}.
\end{align}
In particular, the mean 
$|\Lambda|^{-1}\sum_{j\in\Lambda}(x_j+z_j)=|\Lambda|^{-1}\sum_{j\in\Lambda}e^{u_j}$
has the probability density 
\begin{align}
f_{1,\epsilon}(t)=\sqrt{\frac{\epsilon}{2\pi t^3}}\exp\left(-\frac{\epsilon(t-1)^2}{2t}\right)
, \quad t>0, 
\end{align}
i.e.\ it is an inverse Gaussian distribution with parameters $1$ and $\epsilon$
with respect to $P^\Lambda_{W,h}$.  
\end{corollary}
We now give the proofs of the two results.

\medskip\noindent
\begin{proof}[Proof of Lemma \ref{le:integrate-variables-out}]
It suffices to prove the first equality in \eqref{eq:localization-applied},
because the second one is just the 
special case of $\Lambda$ being a singleton.
Let 
$F:\hthree_+\times(\hthree_+)^\Lambda\to\cA_0$ be defined by
\begin{align}
F(a,v_\Lambda):=&
f\left((\sk{v_i,v_j})_{i,j\in\Lambda}\right)e^{\sum_{j\in\Lambda}c_j\sk{v_j,a}}.
\end{align}
Recall $o$ from equation \eqref{eq:def-o}.

The proof  uses the supersymmetry
operators $L_{ij}$ (cf. Def. \ref{def:supersym-operators} in the appendix), and is organized in four steps, according to the form
of the vector $a.$ Precisely $a=(0,0,z_a,0,0)$ in case 1,  $a=(x_{a},0,z_a,0,0)$ in case 2,
 $a=(x_{a},y_{a},z_a,0,0)$ in case 3 and finally $a$ has the most general form
$a=(x_a,y_a,z_a,\xi_a,\eta_a)$ in case 4. \medskip

\noindent \emph{Case 1.} 
Let $a=\alpha o,$ where the  parameter $\alpha$ is an even element of the Grassmann algebra
with $\alpha>0,$ and must not depend on  the integration variables $v_j$. In this case, by formula
\eqref{eq:Q-annihilates-inner-product} (from Fact \ref{le:L-annihilates-inner-prod}
in the appendix), the integrand $F$ is annihilated by the 
odd supersymmetry operator $Q:=\sum_{j\in\Lambda}Q^{v_j}$, where $Q^{v_j}$ is 
defined in \eqref{eq:def-Q}. 
The hypothesis that $F(a,\cdot)$ is fast decaying allows us to apply
the localization result from  \cite{disertori-spencer-zirnbauer2010}[Appendix C] (in particular 
Proposition 2) as follows 
\begin{align}
&\int_{(H^{2|2})^\Lambda}\cD v_\Lambda \, 
f\left((\sk{v_i,v_j})_{i,j\in\Lambda}\right)e^{\sum_{j\in\Lambda}c_j\sk{v_j,a}}\nonumber\\
&\qquad\qquad  =f\left((\sk{o,o})_{i,j\in\Lambda}\right)e^{\sum_{j\in\Lambda}c_j\sk{o,a}}
=f\left((-1)_{i,j\in\Lambda}\right)e^{-\sum_{j\in\Lambda}c_j\|a\|}. 
\end{align} 
\emph{Case 2.} 
Let $a=(x_a,0,z_a,0,0)$ with $x_a,z_a$ even elements, $z_a>0$,
$\sk{a,a}<0$. We define $\varphi:=\artanh\frac{x_a}{z_a}$, and set, for $0\le t\le 1$, $t\in\R$, 
\begin{align}
a(t)=(x(t),0,z(t),0,0)
:=\sqrt{z_a^2-x_a^2}\  \big (\sinh(t\varphi),0,\cosh(t\varphi),0,0 \big),
\end{align}
where $\sqrt{z_a^2-x_a^2}=\sqrt{-\sk{a,a}} $ is well defined because
$\sk{a,a}<0$. Note that $\sk{a(t),a(t)}=\sk{a,a}$ for all $t$. 
Moreover, 
\begin{align}
\label{eq:a0-a1}
a(0)=\sqrt{z_a^2-x_a^2}(0,0,1,0,0)=\sqrt{z_a^2-x_a^2}\ o\quad\text{ and }\quad a(1)=a. 
\end{align}
By direct computation one has 
\begin{align}
\frac{d}{dt}F(a(t),v_\Lambda)=-\varphi L_{13}^{a(t)}F(a(t),v_\Lambda),
\end{align}
where $L_{13}^{a(t)}$ acts on the first component of $F$, which is then 
substituted by $a(t)$. By formula \eqref{eq:L-annihilates-fn-of-inner-product}
in Fact \ref{le:L-annihilates-inner-prod}, one has
\begin{align}
\Big(L_{13}^{a(t)}+\sum_{j\in\Lambda}L_{13}^{v_j}\Big)F(a(t),v_\Lambda)=0
\end{align}
because the variables appear only inside the inner product. 
We conclude 
\begin{align}
& \frac{d}{dt}\int_{(H^{2|2})^\Lambda}\cD v_\Lambda \, 
F(a(t),v_\Lambda)
=\int_{(H^{2|2})^\Lambda}\cD v_\Lambda \, \frac{d}{dt}F(a(t),v_\Lambda)\nonumber\\
&= -\int_{(H^{2|2})^\Lambda}\cD v_\Lambda \, \varphi L_{13}^{a(t)}F(a(t),v_\Lambda)
=\varphi\int_{(H^{2|2})^\Lambda}\cD v_\Lambda \, 
\sum_{j\in\Lambda}L_{13}^{v_j}F(a(t),v_\Lambda)
=0,
\label{eq:abl-int-0}
\end{align}
where in the last step we have used the Ward identity 
\eqref{eq:ward-identity}. Consequently, the function 
\begin{align}
[0,1]\ni t\mapsto \int_{(H^{2|2})^\Lambda}\cD v_\Lambda \, F(a(t),v_\Lambda)  
\end{align}
is constant and in particular, its values at $t=0$ and $t=1$ agree. 
Recall the expressions \eqref{eq:a0-a1} for 
$a(0)$ and $a(1)$. Using that Case 1 is applicable to $a(0)$, we conclude 
\begin{align}
& \int_{(H^{2|2})^\Lambda}\cD v_\Lambda \, F(a,v_\Lambda) 
=\int_{(H^{2|2})^\Lambda}\cD v_\Lambda \, F(a(1),v_\Lambda) \label{eq:int-exp-sk}
\\
&\quad =\int_{(H^{2|2})^\Lambda}\cD v_\Lambda \, F(a(0),v_\Lambda)
=F\left((-1)_{i,j\in\Lambda}\right)e^{-\sum_{j\in\Lambda}c_j\|a(0)\|}
=F\left((-1)_{i,j\in\Lambda}\right)e^{-\sum_{j\in\Lambda}c_j\|a\|}.
\nonumber
\end{align}
\emph{Case 3:} 
Let $a=(x_a,y_a,z_a,0,0)$ with $x_a,y_a,z_a$ even elements, $z_a>0$,
$\sk{a,a}<0$. Set $\varphi=\artanh\frac{y_a}{z_a}$ and,  for $0\le t\le 1$, $t\in\R$, 
\begin{align}
&a(t)=(x_a,y(t),z(t),0,0)\\
& \text{with}\quad
y(t)=\sqrt{z_a^2-y_a^2}\sinh(t\varphi),\quad 
z(t)=\sqrt{z_a^2-y_a^2}\cosh(t\varphi),
\end{align}
where $\sqrt{z_a^2-y_a^2}$ is well defined because
$z_a^2-y_a^2\ge-\sk{a,a}>0$. Note that $\sk{a(t),a(t)}=\sk{a,a}$ for all $t$. 
Moreover, 
\begin{align}
\label{eq:a0-a1-2}
a(0)=(x_a,0,\sqrt{z_a^2-y_a^2},0,0)\quad\text{ and }\quad a(1)=a. 
\end{align}
As in Case 2, by direct computation one has 
\begin{align}
\frac{d}{dt}F(a(t),v_\Lambda)=-\varphi L_{23}^{a(t)}F(a(t),v_\Lambda)
\end{align}
and again formula \eqref{eq:L-annihilates-fn-of-inner-product} in
Fact \ref{le:L-annihilates-inner-prod} implies
\begin{align}
\Big(L_{23}^{a(t)}+\sum_{j\in\Lambda}L_{23}^{v_j}\Big)F(a(t),v_\Lambda)=0.
\end{align}
Using the same argument as in \eqref{eq:abl-int-0}, with $L_{13}$ replaced
by $L_{23}$, we conclude 
\begin{align}
\frac{d}{dt}\int_{(H^{2|2})^\Lambda}\cD v_\Lambda \, F(a(t),v_\Lambda)=0.
\end{align}
Using that Case 2 is applicable to $a(0)$, the claim follows. \\
\emph{Case 4:} Finally, we consider the most general case 
$a=(x_a,y_a,z_a,\xi_a,\eta_a)$. As in the previous cases, we define for $0\le t\le 1$, 
\begin{align}
&a(t)=(x_a,y_a,z(t),\xi(t),\eta(t))\quad\text{with}\\
&z(t)=z_a-\frac{1-t^2}{z_a}\xi_a\eta_a=\sqrt{z_a^2-(1-t^2)2\xi_a\eta_a}, \quad
\xi(t)=t\xi_a,\quad
\eta(t)=t\eta_a.
\end{align}
Note that $\sk{a(t),a(t)}=\sk{a,a}$ for all $t$. Moreover, 
\begin{align}
\label{eq:a0-a1-3}
a(0)=(x_a,y_a,z_a-\tfrac{1}{z_a}\xi_a\eta_a,0,0)
=(x_a,y_a,\sqrt{z_a^2-2\xi_a\eta_a},0,0)
\quad\text{ and }\quad a(1)=a. 
\end{align}
Below we denote by $\partial_{z(t)}F(a(t),v_\Lambda)$ the derivative 
with respect to the $z$-component of $a(t)$, which is then substituted 
by $z(t)$. The same notational convention holds for $\partial_{\xi(t)}$
and $\partial_{\eta(t)}$. 
Analogue to Cases 2 and 3, using $\xi_az(t)=\xi_az_a$ and
$\eta_az(t)=\eta_az_a$, we get
\begin{align}
&\frac{d}{dt}F(a(t),v_\Lambda)
=\left[z'(t)\partial_{z(t)}
+\xi'(t)\partial_{\xi(t)}
+\eta'(t)\partial_{\eta(t)}\right]F(a(t),v_\Lambda)\nonumber\\
=&\left[\frac{2t}{z_a}\xi_a\eta_a\partial_{z(t)}
+\xi_a\partial_{\xi(t)}
+\eta_a\partial_{\eta(t)}
\right]F(a(t),v_\Lambda)\nonumber\\
=&\frac{1}{z_a}\big[\xi_a\big(z(t)\partial_{\xi(t)}+\eta(t)\partial_{z(t)}\big)
+\eta_a\big(z(t)\partial_{\eta(t)}-\xi(t)\partial_{z(t)}\big)\big]
F(a(t),v_\Lambda)\nonumber\\
=&-\frac{1}{z_a}[\xi_aL_{34}^{a(t)}+\eta_aL_{35}^{a(t)}]F(a(t),v_\Lambda).
\end{align}
Note that $\tfrac{\xi_a}{z_a}$ and $\tfrac{\eta_a}{z_a}$ are the 
analogue of $\varphi$ in cases 2 and 3. 
Formula \eqref{eq:L-annihilates-fn-of-inner-product} implies
\begin{align}
\Big(L_{ik}^{a(t)}+\sum_{j\in\Lambda}L_{ik}^{v_j}\Big)F(a(t),v_\Lambda)=0
\end{align}
for $ik=34,35$. Using again the Ward identity from Fact 
\ref{le:L-annihilates-inner-prod} and the fact that Case 3 is applicable to 
$a(0)$ we conclude. 
\end{proof}

As a first application of Lemma \ref{le:integrate-variables-out}, we determine
the super Laplace transform of $|\Lambda|^{-1}\sum_{j\in\Lambda}v_j$.

\medskip\noindent
\begin{proof}[Proof of Corollary \ref{cor:expectation-mean}]
Using first the definition \eqref{eq:def-superintegration-form} of the 
superintegration form and then Lemma \ref{le:integrate-variables-out} with 
the function $f\left((\sk{v_i,v_j})_{i,j\in\Lambda}\right)=e^{\sum_{i,j\in\Lambda}W_{ij}(1+\sk{v_i,v_j})+\epsilon}$, 
$c_j=|\Lambda|^{-1}$ for all $j$, and $a$ replaced by $a+\epsilon o\in\hthree_+$, 
we obtain 
\begin{align}
E^\Lambda_{W,h}\left[e^{\sk{a,|\Lambda|^{-1}\sum_{j\in\Lambda}v_j}}\right]
=&\int_{(H^{2|2})^\Lambda}\cD v_\Lambda
e^{\sum_{i,j\in\Lambda}W_{ij}(1+\sk{v_i,v_j})+|\Lambda|^{-1}\epsilon\sum_{i\in\Lambda}(1+\sk{v_i,o})}
e^{\sk{a,|\Lambda|^{-1}\sum_{j\in\Lambda}v_j}}\nonumber\\
=&e^{\epsilon-\| a+\epsilon o\|}=\exp(\epsilon-\sqrt{\epsilon^2-2\epsilon\sk{a,o}-\sk{a,a}}). 
\end{align}
We consider $a=(b,0,-b,0,0)$ with $b<0$. Then, we find 
$\sk{a,a}=0$, $\sk{a,o}=b$, $a+\epsilon o\in\hthree_+$ 
and   
$be^{u_j}=b(x_j+z_j)=\sk{a,v_j}$. Consequently, the Laplace transform 
of $|\Lambda|^{-1}\sum_{j\in\Lambda}e^{u_j}=|\Lambda|^{-1}\sum_{j\in\Lambda}(x_j+z_j)$ 
is given by 
\begin{align}
E^\Lambda_{W,h}\left[e^{b|\Lambda|^{-1}\sum_{j\in\Lambda}e^{u_j}}\right]
=E^\Lambda_{W,h}\left[e^{\sk{a,|\Lambda|^{-1}\sum_{j\in\Lambda}v_j}}\right]
=e^{\epsilon-\sqrt{\epsilon^2-2\epsilon b}}. 
\end{align}
This is the Laplace transform of an inverse Gaussian distribution with 
parameters $1$ and $\epsilon$ and the claim follows. 
\end{proof}

\subsection{Reduction to the effective model}

Using the reduction of integrals given in Lemma \ref{le:integrate-variables-out} and a 
representation of the action of the hierarchical $\htwo$-model, 
we prove Theorem \ref{thm:reduced-model}. 

\medskip\noindent
\begin{proof}[Proof of Theorem \ref{thm:reduced-model}]
The conditions on $H_jo+a_j$ guarantee that both integrals in 
\eqref{eq:expectation-reduced-model} are well-defined and finite. 
Let $\cA_h$ denote the set of maximal antichains which are compatible with $h$ and
recall the partial order  $\prec $ for antichains defined in Section \ref{section2.2}.
Since there do not exist infinite 
increasing sequences $A_1\prec A_2\prec A_3\prec\cdots$ in $\cA_h$, we can 
prove the claim by induction over $\cA_h$ with respect to the reversed partial order 
$\succ$ of $\prec$. 

Let $A\in\cA_h$. As induction hypothesis, assume that the claim holds 
for all maximal antichains $A'\in\cA_h$ with $A\prec A'$. 
If $A=\Lambda_N$, there is nothing 
to prove. Otherwise, we find $i\in A\setminus\Lambda_N$ and set 
$A'=(A\setminus\{i\})\cup \B_i$. Note that $A'$ is a maximal antichain compatible with $h$ and
fulfilling  $A\prec A'.$  
Set $a_j'=a_j$ for $j\in A\setminus\{i\}$ and $a_j'=|\B_i|^{-1}a_i$ for 
$j\in \B_i$. Then, $\sum_{j\in \B_i}\sk{a_j',v_j}=\sk{a_i,v_i}$ by the definition 
\eqref{eq:def-block-spin} of the block spin variable $v_i$, and hence 
$\sum_{j\in A}\sk{a_j,v_j}=\sum_{j\in A'}\sk{a_j',v_j}$.
Moreover, for $j\in B_{i},$ the equality $H_{j}=h_{j}$ and $a_{i}+H_{i}o\in \hthree_{+0},\hthree_{+} $
imply
$a'_{j}+h_{j}o\in \hthree_{+0},\hthree_{+},$ respectively. 
Consequently, 
the induction hypothesis is applicable to $A'$ and the $a_j'$. This yields 
\begin{align}
  \label{eq:begin-induction}
E^{\Lambda_N}_{W,h}\left[e^{\sum_{j\in A}\sk{a_j,v_j}}\right]
=E^{\Lambda_N}_{W,h}\left[e^{\sum_{j\in A'}\sk{a_j',v_j}}\right]
=E^{A'}_{W,H}\left[e^{\sum_{j\in A'}\sk{a_j',v_j}}\right].
\end{align}
By the definition \eqref{eq:action-S} of the action 
$S_{W,H}^{A'}(v_{A'})$ with 
vertex set $A'$, weights $W_{ij}$ and pinning $H_j$ given in 
\eqref{eq:extension-def-W} and \eqref{eq:def-H}, respectively, we obtain 
\begin{align}
\label{eq:S-A-prime}
&S_{W,H}^{A'}(v_{A'})
=\sum_{i,j\in A'}W_{ij}(1+\sk{v_i,v_j})+\sum_{j\in A'}H_j(1+\sk{v_j,o})\\
=&S_{W,H}^{A\setminus\{i\}}(v_{A\setminus\{i\}})
+\sum_{j,k\in \B_i}W_{jk}(1+\sk{v_j,v_k})
+2\sum_{\substack{j\in A\setminus\{i\}\\ k\in \B_i}}W_{jk}(1+\sk{v_j,v_k})
+\sum_{k\in \B_i}H_k(1+\sk{v_k,o}).\nonumber
\end{align}
Note that for $j\in A\setminus\{i\}$ and $k\in \B_i$, one has 
$j\wedge k=j\wedge i$ because $j\in A$ and $A$ is an antichain. It follows that 
$W_{jk}=2^{\ell(j)+\ell(k)}w(\ell(j\wedge k))=2^{\ell(j)}w(\ell(j\wedge i))=2^{-\ell(i)}W_{ij},$ which is
independent of $k.$
Furthermore, since $A$ is compatible with $h$, one has the same value $H_k=2^{-\ell(i)}H_i$ 
for all $k\in\B_i.$ Hence, the last two summands in 
\eqref{eq:S-A-prime} can be written as follows 
\begin{equation}
2\sum_{\substack{j\in A\setminus\{i\}\\ k\in \B_i}}W_{jk}(1+\sk{v_j,v_k})
+\sum_{k\in \B_i}H_k(1+\sk{v_k,o})
= C_{i}+\sum_{k\in   \B_i}2^{-\ell(i)} \sk{v_k, V_{i}},
\end{equation}
where
\begin{align}
C_{i}= H_{i}+ \sum_{j\in A\setminus\{i\}} 2W_{ij},\qquad
V_{i}=H_{i}o+ \sum_{j\in A\setminus\{i\}} 2W_{ij}v_{j},
\end{align}
and we used, for each $k\in \B_i,$
\begin{align}
\sum_{j\in A\setminus\{i\}}2W_{jk}v_{j}+H_ko= 2^{-\ell(i)}\Big (H_{i}o+ \sum_{j\in A\setminus\{i\}} 2W_{ij}v_{j} 
\Big )=2^{-\ell(i)}V_i.
\end{align}
Substituting this  in \eqref{eq:S-A-prime} yields
\begin{align}
\label{eq:S-A-prime2}
S_{W,H}^{A'}(v_{A'})=&
S_{W,H}^{A\setminus\{i\}}(v_{A\setminus\{i\}})+C_{i}
+\sum_{j,k\in \B_i}W_{jk}(1+\sk{v_j,v_k})
+\sum_{k\in \B_i}2^{-\ell(i)}\sk{v_k,V_{i}}.
\end{align}
Using first the definition of the superintegration form \eqref{eq:def-superintegration-form} and then
this representation and the fact $a_j'=2^{-\ell(i)}a_i$ for $j\in \B_i$, we obtain
\begin{align}
&E^{A'}_{W,H}\left[e^{\sum_{j\in A'}\sk{a_j',v_j}}\right]
=\int\limits_{(H^{2|2})^{A'}}\cD v_{A'} \, 
e^{S^{A'}_{W,H}(v_{A'})}e^{\sum_{j\in A'}\sk{a_j',v_j}}\nonumber\\
&\qquad =\int\limits_{(H^{2|2})^{A\setminus\{i\}}}\cD v_{A\setminus\{i\}}\,
e^{S_{W,H}^{A\setminus\{i\}}(v_{A\setminus\{i\}})} e^{\sum_{j\in A\setminus\{i\}}\sk{a_j',v_j}}
e^{C_{i}}
\nonumber\\
&\qquad \qquad \cdot\int\limits_{(H^{2|2})^{\B_i}}\cD v_{\B_i} 
\, e^{\sum_{j,k\in \B_i}W_{jk}(1+\sk{v_j,v_k})}
e^{\sum_{k\in \B_i}2^{-\ell(i)}\langle v_k,V_{i}+a_i\rangle}.
\end{align}
We apply Lemma \ref{le:integrate-variables-out} to the inner integral 
with $\Lambda=\B_i$, the function 
$f((\sk{v_j,v_k})_{j,k\in \B_i})=e^{\sum_{j,k\in \B_i}W_{jk}(1+\sk{v_j,v_k})}$, 
$c_j=2^{-\ell(i)}$, and $a=V_{i}+a_i$.
To see that $a\in\hthree_+$ we distinguish two cases. In the case
$W_{ij}=0$ for all $j\in A\setminus\{i\}$, all these points are in different connected components;
hence the assumption on the pinning  guarantees 
$a=H_io+a_i\in\hthree_+$. Otherwise, we have $\sum_{j\in A\setminus\{i\}}W_{ij}v_j\in\hthree_+$,
which implies $a\in\hthree_+$ because of $H_io+a_i\in\hthree_{+0}$.
Lemma \ref{le:integrate-variables-out} allows us to replace the
multidimensional integration with respect to $\cD v_{\B_i}$ by a single
integral with respect to $\cD v_i$ over $\htwo$. 
We obtain
\begin{align}
&E^{A'}_{W,H}\left[e^{\sum_{j\in A'}\sk{a_j',v_j}}\right]
\\
&=\int\limits_{(H^{2|2})^{A\setminus\{i\}}}\cD v_{A\setminus\{i\}}\,
e^{S_{W,H}^{A\setminus\{i\}}(v_{A\setminus\{i\}})} 
e^{\sum_{j\in A\setminus\{i\}}\sk{a_j,v_j}} e^{C_{i}}
\int\limits_{H^{2|2}}\cD v_i 
e^{\sum_{j,k\in \B_i}W_{jk}(1+\sk{v_i,v_i})}
e^{\langle v_i,V_{i}+a_i\rangle}
\nonumber\\
&=\int\limits_{(H^{2|2})^{A\setminus\{i\}}}\cD v_{A\setminus\{i\}}\,
e^{S_{W,H}^{A\setminus\{i\}}(v_{A\setminus\{i\}})} 
\int\limits_{H^{2|2}}\cD v_i \,
e^{2\sum_{j\in A\setminus\{i\}}W_{ij}(1+\sk{v_i,v_j})}e^{H_i(1+\sk{v_i,o})}
   e^{\sum_{j\in A}\sk{a_j,v_j}},\nonumber
\end{align}
where we used $\sk{v_i,v_i}=-1$ for $v_i\in\htwo$ and the definitions of
$C_i$ and $V_i$.
To merge the exponents, we observe that 
\begin{align}
  S_{W,H}^A(v_A)
  =S_{W,H}^{A\setminus\{i\}}(v_{A\setminus\{i\}})+
2\sum_{j\in A\setminus\{i\}}W_{ij}(1+\sk{v_i,v_j})+H_i(1+\sk{v_i,o}).
\end{align}
Together with \eqref{eq:begin-induction} this concludes the induction step: 
\begin{align}
  E^{A'}_{W,H}\left[e^{\sum_{j\in A'}\sk{a_j',v_j}}\right]
  =&\int\limits_{(H^{2|2})^A}\cD v_A\,
e^{S_{W,H}^A(v_A)} e^{\sum_{j\in A}\sk{a_j,v_j}}
=E^A_{W,H}\left[e^{\sum_{j\in A}\sk{a_j,v_j}}\right].
\end{align}
\end{proof}

Using the reduction to the effective model, we can prove the identities 
in distribution for certain horospherical coordinates as stated 
in Corollary \ref{cor:distribution-effective-model}. 

\medskip\noindent 
\begin{proof}[Proof of Corollary \ref{cor:distribution-effective-model}]
First we prove claim \eqref{eq:identity-laplace-trafos} under the extra 
assumption that $h_i>0$ for all $i\in\Lambda_N$. 
It suffices to show that the joint Laplace transforms of 
$(|\B_j|^{-1}\sum_{i\in \B_j}e^{u_i},|\B_j|^{-1}\sum_{i\in \B_j}s_ie^{u_i})_{j\in A}$ 
and $(e^{u_j},s_je^{u_j})_{j\in A}$ with respect to $P^{\Lambda_N}_{W,h}$ and   
$P^A_{W,H}$, respectively, agree on a non-empty open set and
are finite there. Thus, we will show the following equality for all 
$(c_j)_{j\in A}\in(0,\infty)^A$ and all $c_j'\in(-H_j,H_j)$, $j\in A$:
\begin{align}
\label{eq:claim-laplace-trafo-mean-along-antichain}
  E^{\Lambda_N}_{W,h}\left[e^{-\sum_{j\in A}|\B_j|^{-1}\sum_{i\in \B_j}(c_j e^{u_i}+c_j' s_ie^{u_i})}\right]
=E^A_{W,H}\left[e^{-\sum_{j\in A}(c_j e^{u_j}+c_j's_je^{u_j})}\right]. 
\end{align}
We rewrite this claim in cartesian coordinates. 
Using $-c_j e^{u_i}-c_j's_ie^{u_i}=-c_j(x_i+z_i)-c_j'y_i=\sk{v_i,a_j}$
with $a_j=(-c_j,-c_j',c_j,0,0)$ for $j\in A$, $i\in \B_j$ on the left-hand side
and $-c_j e^{u_j}-c_j's_je^{u_j}=\sk{v_j,a_j}$ for $j\in A$ on the right-hand side, 
it takes the following form: 
\begin{align}
  \label{eq:claim-laplace-trafo-mean-along-antichain-reformulated}
E^{\Lambda_N}_{W,h}\left[e^{\sum_{j\in A}\sk{v_j,a_j}}\right]
=  E^{\Lambda_N}_{W,h}\left[e^{\sum_{j\in A}|\B_j|^{-1}\sum_{i\in \B_j}\sk{v_i,a_j}}\right]
=E^A_{W,H}\left[e^{\sum_{j\in A}\sk{v_j,a_j}}\right]. 
\end{align}
Since $\sk{H_jo+a_j,H_jo+a_j}=c_j'^2-2c_jH_j-H_j^2<0$, $H_j+c_j>0$, and thus 
$H_jo+a_j\in\hthree_+$, the claim, including finiteness of the expectations in 
\eqref{eq:claim-laplace-trafo-mean-along-antichain-reformulated}, follows from Theorem 
\ref{thm:reduced-model}.

Now we drop the extra assumption $h_i>0$ for all $i\in\Lambda_N$. 
Let $\cN:=\{i\in\Lambda_N:h_i>0\}$,  
\begin{align}
\cH:=\{ (h_i')_{i\in\Lambda_N}: h_i'=h_i\text{ for }i\in\cN, 0< h_i'\le 1  
\text{ for }i\notin\cN, h'\text{ compatible with }A\}.
\end{align}
The probability density of $P^{\Lambda_N}_{W,h'}(du_{\Lambda_N}\, ds_{\Lambda_N})$
specified in \eqref{eq:P-in-horosph-coor} is dominated by  
\begin{align}
c\, e^{-\frac12\sum_{i,j\in\Lambda_N}W_{ij}(B_{ij}-1)}
e^{-\sum_{i\in\Lambda_N}h_i(B_i-1)}p((e^{u_i})_{i\in\Lambda_N})
\prod_{i\in\Lambda_N}\frac{e^{-u_i}}{2\pi}
\end{align}
with some numerical polynomial $p$ with positive coefficients and a constant 
$c=c((h_i)_{i\in\cN})$, uniformly in $h'\in\cH$. This dominating function 
is integrable by our assumption on the pinning strength function $h$, cf.\
remarks below formula \eqref{eq:E-Lambda-W}. 
Hence, by dominated convergence, we obtain 
\begin{align}
  \label{eq:limit1}
\lim_{h'\downarrow h}E^{\Lambda_N}_{W,h'}[f(u_{\Lambda_N},s_{\Lambda_N})]
=  E^{\Lambda_N}_{W,h}[f(u_{\Lambda_N},s_{\Lambda_N})]
\end{align}
for every bounded measurable function $f:(\R^2)^{\Lambda_N}\to\R$. 
The same argument applies 
when we replace $\Lambda_N$ by $A$ and $h$ by $H|_A$. Note that $H|_A$ 
inherits the assumption on the pinning strength function phrased 
below formula \eqref{eq:E-Lambda-W}. This yields 
\begin{align}
  \label{eq:limit2}
\lim_{H'\downarrow H}E^A_{W,H'}[f(u_A,s_A)]
=  E^A_{W,H}[f(u_A,s_A)]
\end{align}
for every bounded measurable function $f:(\R^2)^A\to\R$. The left-hand sides
of \eqref{eq:limit1} and \eqref{eq:limit2} agree for $h'\in\cH$ and the
induced $H'$, which allows us to conclude.  

The first claim in Corollary \ref{cor:distribution-effective-model} is an 
immediate consequence since $(e^{u_j},s_je^{u_j})_{j\in B}$ uniquely determines
$(u_j,s_j)_{j\in B}$ and $\B_j=\{j\}$ holds for $j\in B=A\cap\Lambda_N$.
\end{proof}

\section{Tightness}

To prove Theorem \ref{thm:main} we use a slight reformulation of a bound
from Lemma 4 in \cite{disertori-spencer-zirnbauer2010}. This will be an
essential ingredient of the argument. Recall the abbreviation 
$B_{ij}$ introduced in \eqref{eq:def-Bij}. 

\begin{lemma}[Alignment of two spins]
\label{le:estimate-Bij}
Consider the $\htwo$-model on a finite vertex set $\Lambda$ with 
arbitrary weights $W$ and pinning $h$.
For $W_{ij}>0$ and $\delta>W_{ij}^{-1}$, the following bound holds
\begin{align}
P^\Lambda_{W,h}(B_{ij}\ge 1+\delta)\le W_{ij}\delta e^{-(W_{ij}\delta-1)}.
\end{align}
\end{lemma}
\begin{proof}
The proof is based on Lemmas 3 and 4 in \cite{disertori-spencer-zirnbauer2010},
which generalize to non-constant weights $W_{ij}$ in a straightforward
way. Using first Chebyshev's inequality and then Lemma 3 
in \cite{disertori-spencer-zirnbauer2010}, we obtain for all $\delta>0$
and all $\gamma\in(0,1)$
\begin{align}
P^\Lambda_{W,h}(B_{ij}\ge 1+\delta)
=E^\Lambda_{W,h}[1_{\{B_{ij}\ge 1+\delta\}}]
\le E^\Lambda_{W,h}[e^{W_{ij}\gamma (B_{ij}-1-\delta)}]
\le (1-\gamma)^{-1}e^{-W_{ij}\gamma\delta}.
\label{eq:upper-bound-Bij-DSZ10}
\end{align}
The upper bound is of the form $e^{-f(\gamma)}$ with 
$f(\gamma)=\log(1-\gamma)+\alpha\gamma$ and $\alpha=W_{ij}\delta$. 
In order to optimize this upper bound,
we maximize $f$. Observe that $f'(\gamma)=\alpha-(1-\gamma)^{-1}$, which equals
zero iff $\gamma=\gamma_c=1-\alpha^{-1}$. By assumption $\alpha>1$, and hence 
the critical point $\gamma_c$ belongs to $(0,1)$. Using 
$f(1-\alpha^{-1})=-\log\alpha+\alpha-1$, the claim follows. 
\end{proof}

Using Lemma \ref{le:estimate-Bij} and the reduction to the effective model,
we prove our main result. 

\medskip\noindent\begin{proof}[Proof of Theorem \ref{thm:main}]
We treat the two cases, uniform pinning (``U'') and pinning at one point 
(``1P'') simultaneously. 

Let $\rho>0$. It suffices to find $M=M(\rho)>0$ such that for all $N\in\N$,
$\epsilon>0$, and $p,q\in\Lambda_N$ with $p\neq q$, one has 
\begin{align}
\label{eq:U}
\text{case U: }\quad &P^{\Lambda_N}_{W,\epsilon 1}(|u_p-u_q|\ge M)\le\rho,\\ 
\label{eq:1P}
\text{case 1P: }\quad &P^{\Lambda_N}_{W,\epsilon\delta_p}(|u_p-u_q|\ge M)\le\rho. 
\end{align}
Note that in \eqref{eq:1P} the vertex $p$ appears twice, as pinning point
and in the variable $u_p$. This suffices indeed because for $i_0,p,q\in\Lambda_N$
we have the estimate 
\begin{align}
  P^{\Lambda_N}_{W,\epsilon\delta_{i_0}}(|u_p-u_q|\ge M)
  \le  P^{\Lambda_N}_{W,\epsilon\delta_{i_0}}\left(|u_{i_0}-u_p|\ge\frac{M}{2}\right)+
   P^{\Lambda_N}_{W,\epsilon\delta_{i_0}}\left(|u_{i_0}-u_q|\ge\frac{M}{2}\right).
\end{align}
Let $N\in\N$,
$\epsilon>0$, and $p,q\in\Lambda_N$ with $p\neq q$. Recall the notation introduced at the
beginning of Section \ref{subsec:hierarchical-h22}. 
We abbreviate $d=d(p,q)\ge 1$. 
Let $\overline 0:=1$, $\overline 1:=0$. 
For $j=(j_0,\ldots,j_{|j|-1})\in\T^N\setminus\{\emptyset\}$, let 
$\overline j:=(\overline j_0,j_1,j_2,\ldots,j_{|j|-1})$ be the nearest 
neighbor of $j$ on the same level. 
We consider the set $A\subseteq\T^N$ given by 
\begin{align}
  \label{eq:def-A}
A:=&\{p,q\}\cup
     \{j\in\T^N\setminus\{\emptyset\}:j\not\preceq p, j\not\preceq q,(\cut(j)\preceq p\text{ or }
     \cut(j)\preceq q)\}\\
=&\{p,q\}\cup\{\overline{\cut^l(p)}:0\le l\le N-1, l\neq d-1\}
\cup\{\overline{\cut^l(q)}:0\le l\le d-2\}.\nonumber
\end{align}
An illustration is given in Figure \ref{fig:antichain}. 

\begin{figure}
  \centering
\includegraphics[width=0.8\textwidth]{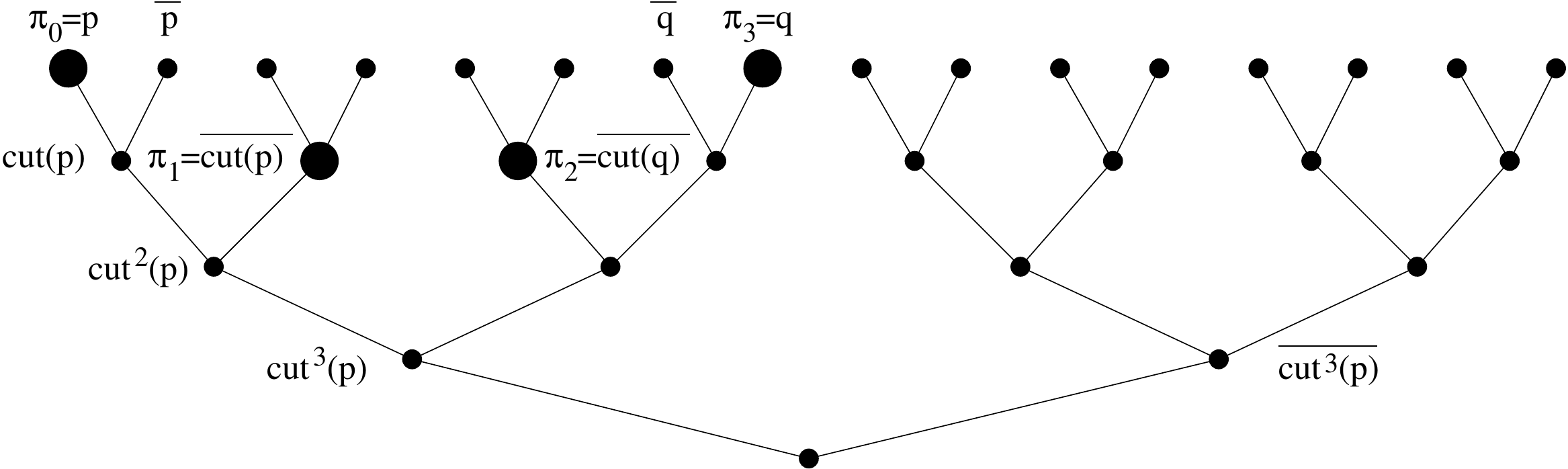}
  \caption{The vertices $p,q\in\Lambda_4$ have hierarchical distance 
$d=3$. The maximal antichain from \eqref{eq:def-A} equals 
$A:=\{p,q,\overline p,\overline q,\overline{\cut(p)},\overline{\cut(q)},\overline{\cut^3(p)}\}$. The path $\pi$ in \eqref{eq:def-pi} is given 
by $\pi=(p,\overline{\cut(p)},\overline{\cut(q)},q)$.}
  \label{fig:antichain}
\end{figure}

Note that $A$ is a maximal antichain. Furthermore, it is compatible with 
both pinning functions, $h=\epsilon 1$ (case U) and $h=\epsilon\delta_p$
(case 1P). We remark that this is the reason why we need $p$ as a pinning point 
in formula \eqref{eq:1P}. 

We consider the path $\pi=(\pi_l)_{0\le l\le m}$ in the antichain $A$ 
from $\pi_0=p$ to $\pi_m=q$ given by the concatenation of $p$, 
$(\overline{\cut^l(p)})_{1\le l\le d-2}$, the reversed path of 
$(\overline{\cut^l(q)})_{1\le l\le d-2}$, and $q$. In other words,
for $m=d=1$, we have $\pi=(p,q)$ and for $d\ge 2$, we have $m=2d-3$, and 
\begin{align}
  \label{eq:def-pi}
\pi_l=\left\{
  \begin{array}{ll}
p &\text{for } l=0,\\
    \overline{\cut^l(p)} & \text{for } 1\le l\le d-2,\\
    \overline{\cut^{m-l}(q)} & \text{for } d-1\le l\le m-1, \\
q &\text{for } l=m.
  \end{array}
\right.
\end{align}
Note that we exclude on purpose the points $\overline p=\overline{\cut^0(p)}$
and $\overline q=\overline{\cut^0(q)}$ in the second and third line of
\eqref{eq:def-pi} (though they belong to the antichain) since they do not
help in the estimates below. 

In the case $d\ge 2$, the edges $\{\pi_l,\pi_{l+1}\}$ (with $l=0,\ldots,m-1$) of the path
$\pi$ have the weights 
\begin{align}
\beta_l':=
W_{\pi_l\pi_{l+1}}=&2^{\ell(\pi_l)+\ell(\pi_{l+1})}w(\ell(\pi_l\wedge \pi_{l+1}))\nonumber\\
=&\left\{
  \begin{array}{ll}
2^{2l+1}w(l+2)=\beta_l   & \text{if } 0\le l\le d-3,\\
2^{2d-4}w(d)
=\frac12\beta_{d-2}   & \text{if } l=d-2,\\
2^{2(m-l)-1}w(m-l+1)=\beta_{m-l-1}    & \text{if } d-1\le l\le m-1.
  \end{array}
                                        \right.
\label{eq:def-beta-prime}                                        
\end{align}
In the case $d=1$, the single edge $\{p,q\}$ of the path $\pi$ has the weight
$\beta_0'':=W_{pq}=w(1)$. We choose $\delta_0''>0$ large enough such that
$\beta_0''\delta_0''>1$ and $\beta_0''\delta_0'' e^{-(\beta_0''\delta_0''-1)}<\rho$.
We show below the following claim.\\
\textit{Claim: }There is a sequence $(\delta_l)_{l\in\N_0}$ of positive numbers
    (independent of $p$ and $q$) such that  
\begin{align}
  \beta_l\delta_l>1\text{ for all }l,\quad
  \sum_{l=0}^\infty\arcosh(1+\delta_l)<\infty
\quad\mbox{and}\quad
\sum_{l=0}^\infty\beta_l\delta_le^{-\beta_l\delta_l+1}<\frac{\rho}{2}.
\label{eq:sum-arcosh-finite}
\end{align}
Using this claim, we show first that the estimates \eqref{eq:U} and \eqref{eq:1P} hold. 
We define 
\begin{align}
M:=\max\Big\{\sup_{d'\ge 2}\Big( 
  2\sum_{l=0}^{d'-3}\arcosh(1+\delta_l)+\arcosh(1+2\delta_{d'-2})\Big),
  \arcosh(1+\delta_0'')\Big\}.
  \label{eq:def-M}
\end{align}
As a consequence of \eqref{eq:sum-arcosh-finite}, the sequence 
$(\delta_l)_l$ is bounded and $M$ is finite.

We consider now the case $d\ge 2$. 
We introduce 
\begin{align}
  \label{eq:def-delta-strich}
\delta_l'=&\left\{
  \begin{array}{ll}
\delta_l   & \text{for } 0\le l\le d-3,\\
2\delta_{d-2}   & \text{for } l=d-2,\\
\delta_{m-l-1}    & \text{for } d-1\le l\le m-1.
  \end{array}
\right.  
\end{align}
Note that 
\begin{align}
\sum_{l=0}^{m-1}\arcosh(1+\delta_l')=
  2\sum_{l=0}^{d-3}\arcosh(1+\delta_l)+\arcosh(1+2\delta_{d-2})
\le M.
\label{eq:lower-bound-M}
\end{align}
Since $A$ is a maximal antichain containing $p$ and $q$, the distributions 
of $u_p-u_q$ with respect to $P^{\Lambda_N}_{W,h}$ and $P^A_{W,H}$ coincide by 
Corollary \ref{cor:distribution-effective-model}; here $H$ is the 
extension of the pinning function $h$ given in \eqref{eq:def-H}. 
Using \eqref{eq:lower-bound-M} in the first inequality, we estimate 
\begin{align}
P^{\Lambda_N}_{W,h}(|u_p-u_q|\ge M)
=&P^A_{W,H}(|u_{\pi_0}-u_{\pi_m}|\ge M)  \nonumber\\
\le &P^A_{W,H}\Big(|u_{\pi_0}-u_{\pi_m}|\ge \sum_{l=0}^{m-1}\arcosh(1+\delta_l')
\Big)  
\nonumber\\
\le & P^A_{W,H}\Big(\sum_{l=0}^{m-1}|u_{\pi_l}-u_{\pi_{l+1}}|\ge \sum_{l=0}^{m-1}\arcosh(1+\delta_l')\Big)  \nonumber\\
\le & \sum_{l=0}^{m-1} P^A_{W,H}(|u_{\pi_l}-u_{\pi_{l+1}}|\ge\arcosh(1+\delta_l')) 
\nonumber\\
= & \sum_{l=0}^{m-1} P^A_{W,H}(\cosh(u_{\pi_l}-u_{\pi_{l+1}})\ge 1+\delta_l')  .
\label{eq:sum-P-cosh}
\end{align}
Recall the notation $B_{ij}$ from formula \eqref{eq:def-Bij}. In particular, 
$\cosh(u_i-u_j)\le B_{ij}$ for $i,j\in A$. To estimate the $l$-th 
summand on the right-hand side of formula \eqref{eq:sum-P-cosh}
we observe that $W_{\pi_l\pi_{l+1}}=\beta_l'$ and $\beta_l'\delta_l'=\beta_l\delta_l>1$ 
as a consequence of \eqref{eq:def-beta-prime}, \eqref{eq:def-delta-strich}, and Claim \eqref{eq:sum-arcosh-finite}
so that Lemma \ref{le:estimate-Bij} is applicable and yields 
\begin{align}
P^A_{W,H}(\cosh(u_{\pi_l}-u_{\pi_{l+1}})\ge 1+\delta_l')  
\le  P^A_{W,H}(B_{\pi_l\pi_{l+1}}\ge 1+\delta_l')  
\le \beta_l'\delta_l' e^{-(\beta_l'\delta_l'-1)}.
\end{align}
Summing over $l$ and using Claim \eqref{eq:sum-arcosh-finite} in the last step, 
we conclude 
\begin{align}
&P^{\Lambda_N}_{W,h}(|u_p-u_q|\ge M)
\le \sum_{l=0}^{m-1}\beta_l'\delta_l' e^{-(\beta_l'\delta_l'-1)}\nonumber\\
= & 2\sum_{l=0}^{d-3}\beta_l\delta_l e^{-(\beta_l\delta_l-1)}
+\beta_{d-2}\delta_{d-2} e^{-(\beta_{d-2}\delta_{d-2}-1)}
\le 2\sum_{l=0}^\infty\beta_l\delta_l e^{-(\beta_l\delta_l-1)}<\rho
.
\end{align}
For $d=1$, using $\arcosh(1+\delta_0'')\le M$ from \eqref{eq:def-M} and
again  Lemma \ref{le:estimate-Bij}, it also holds
\begin{align}
  P^{\Lambda_N}_{W,h}(|u_p-u_q|\ge M)\le P^{\Lambda_N}_{W,h}(\cosh(u_p-u_q)\ge 1+\delta_0'')
  \le \beta_0''\delta_0'' e^{-(\beta_0''\delta_0''-1)}<\rho. 
\end{align}

It remains to show Claim \eqref{eq:sum-arcosh-finite}. Assume that the hypothesis
\eqref{eq:assumption-main-thm} of Theorem \ref{thm:main}
holds. To ensure $\sum_{l=0}^\infty\arcosh(1+\delta_l)<\infty$, it is necessary
that $\delta_l\to 0$ as $l\to\infty$. On the other hand, to ensure
$\sum_{l=0}^\infty\beta_l\delta_le^{-\beta_l\delta_l+1}<\frac{\rho}{2}$, we will 
choose $\delta_l$ such that $\beta_l\delta_l\to\infty$ fast enough as $l\to\infty$ and 
$\beta_l\delta_l$ is large enough for a sufficiently large initial piece 
$l=0,\ldots,l_1$. More precisely, we proceed as follows. 

From assumption
\eqref{eq:assumption-main-thm} we know $\lim_{l\to\infty}\beta_l=\infty$. 
Take $l_0\in\N$ so large that $\log\log\sqrt{\beta_l}$ is well-defined
and positive for all $l\ge l_0$. For these $l$, we set 
\begin{align}
\label{eq:def-tilde-delta-l}
\tilde\delta_l:=
\frac{1}{\beta_l}\Big(1+\log\sqrt{\beta_l}
+\frac32\log\log\sqrt{\beta_l}\Big)>\frac{1}{\beta_l}.
\end{align}
We observe that 
\begin{align}
&\sum_{l=l_0}^\infty\beta_l\tilde\delta_le^{-\beta_l\tilde\delta_l+1}
=\sum_{l=l_0}^\infty\Big(1+\log\sqrt{\beta_l}+\frac32\log\log\sqrt{\beta_l}\Big)
\exp\Big(-\log\sqrt{\beta_l}-\frac32\log\log\sqrt{\beta_l}\Big)
\nonumber\\
=&\sum_{l=l_0}^\infty\Big(1+\log\sqrt{\beta_l}+\frac32\log\log\sqrt{\beta_l}\Big)
\frac{1}{\sqrt{\beta_l}}(\log\sqrt{\beta_l})^{-\frac32}
   <\infty.
   \label{eq:calculation}
\end{align}
The last estimate used assumption \eqref{eq:assumption-main-thm}.
However, we remark that its full strength is not used here, but only below. 
The estimate \eqref{eq:calculation} implies that there exists $l_1\ge l_0$ such that 
\begin{align}
   \sum_{l=l_1}^\infty\beta_l\tilde\delta_le^{-\beta_l\tilde\delta_l+1}<\frac{\rho}{4}.
\end{align}
We choose $\delta_l>\beta_l^{-1}$ for $0\le l\le l_1-1$ large enough, such that 
\begin{align}
   \sum_{l=0}^{l_1-1}\beta_l\delta_le^{-\beta_l\delta_l+1}<\frac{\rho}{4}.
\end{align}
For $l\ge l_1$, we set $\delta_l:=\tilde\delta_l$. Then, the condition for 
the second series in \eqref{eq:sum-arcosh-finite} follows.
Observe that $\beta_l\delta_l>1$ for all $l\in\N_0$ by  
\eqref{eq:def-tilde-delta-l} and our choice of $\delta_l$ for small $l$. 

It remains to prove the inequality for the first series
in \eqref{eq:sum-arcosh-finite}, which is equivalent to the claim 
\begin{align}
\sum_{l=l_1}^\infty\arcosh(1+\delta_l)<\infty.
\end{align}
We estimate the $l$-th summand in the last series:
\begin{align}
\arcosh(1+\delta_l)\le &\sqrt{2\delta_l}
=\sqrt{\frac{2}{\beta_l}}\sqrt{1+\log\sqrt{\beta_l}
+\frac32\log\log\sqrt{\beta_l}}
=O\left(\sqrt{\frac{\log\beta_l}{\beta_l}}\right)
\end{align}
as $l\to\infty$, which is summable over $l$ by hypothesis \eqref{eq:assumption-main-thm}.
This completes the proof of Claim \eqref{eq:sum-arcosh-finite} and 
hence the proof of the theorem. 
\end{proof}

\begin{remark}
\label{rem:exp-complexity}
In the proof of Theorem \ref{thm:main} it is shown that for $p,q\in\Lambda_N$
with $d(p,q)=d$, the set 
\begin{align}
A:=\{p,q\}\cup\{\overline{\cut^l(p)}:0\le l\le N-1, l\neq d-1\}
\cup\{\overline{\cut^l(q)}:0\le l\le d-2\}
\end{align}
is a maximal antichain containing $p$ and $q$; see Figure 
\ref{fig:antichain} for an illustration. 
In order to study the distribution of $(u_p,u_q)$ with respect to 
$P^{\Lambda_N}_{W,h}$ with a pinning function $h$ such that $A$ is
compatible with it, 
cf.\ Figure \ref{fig:antichain2}, 
one can equally well study its distribution with respect 
to the reduced model $P^A_{W,H}$. Note that $|\Lambda_N|=2^N$, whereas 
$|A|=N+d$, resulting in an exponential decrease of complexity. More generally, 
one can construct a 
maximal antichain $A(B)$ containing a finite set $B\subseteq\Lambda_N$ as follows. 
The set $A(B)$ consists of $B$ and all vertices which are minimal elements with 
respect to $\preceq$ in the set $\{i\in\T^N:i\not\preceq j\text{ for all }
j\in B\}$.
For fixed $|B|$, as $N\to\infty$, the size of $A(B)$ is again bounded by $O(N)$. 
\end{remark}

\begin{appendix}
\section{Appendix: Grassmann algebras and supersymmetry operators}

\paragraph{Grassmann algebras.}
Here we give some more details on the Grassmann algebra $\cA$ underlying the 
$\htwo$-model with vertex set $\Lambda$. 
We consider a real vector space $V$ of finite dimension $K$. We assume
that $V$ has a basis $\B=\{\chi_1,\ldots,\chi_K\}$ consisting of 
$\xi_i$, $\eta_i$ for $i\in\Lambda$ 
and possibly additional Grassmann parameters. Then the Grassmann algebra 
$\cA$ is defined as the direct sum of exterior product spaces 
$\mathsf{\Lambda}^kV$ of $V$: 
$\cA=\bigoplus_{k=0}^K\mathsf{\Lambda}^kV=\cA_0\oplus\cA_1$ with the even 
(commuting) subalgebra 
$\cA_0=\bigoplus_{k=0}^{\lfloor K/2\rfloor}\mathsf{\Lambda}^{2k}V$
and the odd (anticommuting) subspace 
$\cA_1=\bigoplus_{k=0}^{\lfloor (K-1)/2\rfloor}\mathsf{\Lambda}^{2k+1}V$. 
Note that the body of an even element is the projection to the $0$th-component:
$\body:\cA_0\to\mathsf{\Lambda}^0V=\R$.
Every element $\omega$ in the Grassmann algebra $\cA$ can be uniquely written 
as follows:
\begin{align}
\label{eq:representation-omega}
\omega=\sum_{I\subseteq\{1,\ldots,K\}}c_I\chi^I,   
\end{align}
with coefficients $c_I\in\R$ and basis elements $\chi^I:=\chi_{i_1}\ldots\chi_{i_l}$ 
for $I=\{i_1,\ldots,i_l\}$ with $i_1<\cdots <i_l$ and $\chi^\emptyset=1$. 

\paragraph{Grassmann derivatives.}
For every basis element $\chi\in \B$ and every element $\omega\in\cA$,
we have a unique decomposition $\omega=\omega_0+\chi\omega_1$,
where $\omega_0$ and $\omega_1$ belong to the subalgebra of $\cA$ generated by
the elements of $\B\setminus\{\chi\}$. We define the Grassmann derivative
\begin{align}
\partial_\chi\omega:=\omega_1. 
\end{align}
Note that this derivative depends not only on $\chi$ and $\omega$, but 
also on the choice of the basis $\B$. This is in analogy to the ordinary 
partial derivative, which also depends on the choice of other variables
which are kept fixed.

\paragraph{Supersymmetry operators.}
We review now the (super)symmetry operators of the $\htwo$-model. 
For $n\in\N$, let $[n]:=\{1,2,\ldots,n\}$. We define the parity function 
\begin{align}
p:[5]\to\{0,1\},\quad p(i)=\left\{
  \begin{array}{ll}
    0 & i\in[3],\\
    1 & i\in\{4,5\}. 
  \end{array}\right.
\end{align}
For $i,j\in[5]$, we set 
\begin{align}
\sigma_{ij}:=(-1)^{p(i)+p(j)+p(i)p(j)}=\left\{  
  \begin{array}{ll}
    1 & i,j\in[3],\\
    -1 & \text{otherwise.}
  \end{array}\right.
\end{align}
Using the matrix
\begin{align}
g=(g_{ij})_{i,j\in[5]}:=\left(
  \begin{array}{rrr|rr}
    1 & 0 & 0 & 0 & 0\\
    0 & 1 & 0 & 0 & 0 \\
    0 & 0 & -1 & 0 & 0 \\ \hline
    0 & 0 & 0 & 0 & 1 \\
    0 & 0 & 0 & -1 & 0
  \end{array}\right)
\end{align}
the inner product of $v=(x,y,z,\xi,\eta)$ and $v'=(x',y',z',\xi',\eta')$
reads 
\begin{align}
\sk{v',v}=v'gv^t. 
\end{align}
We write $v=(v^i)_{i\in[5]}$ and $\partial_{v^i}=\frac{\partial}{\partial v^i}$.

\begin{definition}[Supersymmetry operators]
\label{def:supersym-operators}
For $i,j\in[5]$, we define the supersymmetry operators 
  \begin{align}
    \label{eq:def-Lij}
    L_{ij}=L_{ij}^v:=\sum_{k\in[5]}[v^kg_{kj}\partial_{v^i}
    -v^k\sigma_{ij}g_{ki}\partial_{v^j}]
  \end{align}
\end{definition}

We observe that $L_{ij}=-\sigma_{ji}L_{ji}$ for all $i,j$. Furthermore, 
$L_{11}=L_{22}=L_{33}=0$, 
\begin{itemize}
\item $L_{12}=y\partial_x-x\partial_y$ generates Euclidean rotations in
  the $x$-$y$-plane,
\item $-L_{13}=z\partial_x+x\partial_z$ and $-L_{23}=z\partial_y+y\partial_z$
generate Lorentz boosts in the $x$-$z$- and $y$-$z$-plane, respectively,
\item $L_{44}=-2\eta\partial_\xi$, $L_{45}=\xi\partial_\xi-\eta\partial_\eta$,
and $L_{55}=2\xi\partial_\eta$ are even operators, and 
\item $L_{14}=x\partial_\xi-\eta\partial_x$,  
$L_{24}=y\partial_\xi-\eta\partial_y$,  
$L_{15}=\xi\partial_x+x\partial_\eta$, 
$L_{25}=\xi\partial_y+y\partial_\eta$,  
$L_{34}=-(\eta\partial_z+z\partial_\xi)$, and 
$L_{35}=\xi\partial_z-z\partial_\eta$ are odd operators. 
\end{itemize}
We define 
\begin{align}
\label{eq:def-Q}
Q^v:=L_{15}-L_{24}
=x\partial_{\eta}-y\partial_{\xi}+\xi\partial_{x}+\eta\partial_{y}. 
\end{align}

We collect some basic facts on these supersymmetry operators. For more
details see Section 4 and Appendices B and C of 
\cite{disertori-spencer-zirnbauer2010}. 

\begin{fact}[Consequences of supersymmetry]
\label{le:L-annihilates-inner-prod} \hspace{2cm}

\begin{enumerate}
\item For $i,j\in[5]$, the following holds.
\begin{enumerate}
\item The operator $L_{ij}$ annihilates the inner product as follows
\begin{align}
\label{eq:L-annihilates-inner-product}
(L_{ij}^v+L_{ij}^{v'})\sk{v,v'}=0\quad\text{for all }v,v'\in\hthree.
\end{align}
Moreover, for $k\in\N$ and any smooth superfunction $f:\cA_0^{k^2}\to\cA_0$, 
as $v_1,\ldots,v_k$ run over $\hthree$, one has 
\begin{align}
\label{eq:L-annihilates-fn-of-inner-product}
\sum_{l=1}^k L_{ij}^{v_l}f((\sk{v_j,v_{j'}})_{1\le j,j'\le k}) =0.
\end{align}
\item \emph{Ward identity.} For any smooth superfunction 
$f:(\hthree_+)^k\to\cA_0$ with $k\in\N$ such that $f$ is fast decaying on 
$(\htwo)^k$, one has 
\begin{align}
\label{eq:ward-identity}
\int_{(\htwo)^k}\cD v_1\ldots\cD v_k\, \sum_{l=1}^k L_{ij}^{v_l}f(v_1,\ldots,v_k)=0. 
\end{align}
\end{enumerate}
\item The operator $Q$ annihilates the following inner products 
\begin{align}
\label{eq:Q-annihilates-inner-product}
(Q^v+Q^{v'})\sk{v,v'}=0, \quad Q^v\sk{v,o}=0,
\quad\text{for all }v,v'\in\hthree.
\end{align}
\end{enumerate}
\end{fact}

\end{appendix}

\paragraph{Acknowledgements.}
The authors were supported by the DFG Priority Programme 2265 Random 
Geometric Systems.


\begin{thebibliography}{BCHS21}

\bibitem[BCH21]{bauerschmidt-crawford-helmuth-percolation-transition2021}
R.~Bauerschmidt, N.~Crawford, and T.~Helmuth.
\newblock Percolation transition for random forests in $d\ge 3$.
\newblock preprint, arXiv:2107.01878, 2021.

\bibitem[BCHS21]{bauerschmidt-crawford-helmuth-swan-spanning-forests2021}
R.~Bauerschmidt, N.~Crawford, T.~Helmuth, and A.~Swan.
\newblock Random spanning forests and hyperbolic symmetry.
\newblock {\em Comm. Math. Phys.}, 381(3):1223--1261, 2021.

\bibitem[BH21]{bauerschmidt-helmuth-survey2021}
R.~Bauerschmidt and T.~Helmuth.
\newblock Spin systems with hyperbolic symmetry: a survey.
\newblock preprint, arXiv:2109.02566, 2021.

\bibitem[BHS21]{bauerschmidt-helmuth-swan-dynkin-isomorphism2021}
R.~Bauerschmidt, T.~Helmuth, and A.~Swan.
\newblock The geometry of random walk isomorphism theorems.
\newblock {\em Ann. Inst. Henri Poincar\'{e} Probab. Stat.}, 57(1):408--454,
  2021.

\bibitem[Cra21]{crawford2021}
N.~Crawford.
\newblock Supersymmetric {H}yperbolic {$\sigma$}-{M}odels and {B}ounds on
  {C}orrelations in {T}wo {D}imensions.
\newblock {\em J. Stat. Phys.}, 184(3):Paper No. 32, 2021.

\bibitem[DMR17]{disertori-merkl-rolles2017}
M.~Disertori, F.~Merkl, and Silke~W.W. Rolles.
\newblock A supersymmetric approach to martingales related to the
  vertex-reinforced jump process.
\newblock {\em ALEA Lat. Am. J. Probab. Math. Stat.}, 14(1):529--555, 2017.

\bibitem[DMR19]{disertori-merkl-rolles2019}
M.~Disertori, F.~Merkl, and S.W.W. Rolles.
\newblock Martingales and some generalizations arising from the supersymmetric
  hyperbolic sigma model.
\newblock {\em ALEA Lat. Am. J. Probab. Math. Stat.}, 16(1):179--209, 2019.

\bibitem[DS10]{disertori-spencer2010}
M.~Disertori and T.~Spencer.
\newblock Anderson localization for a supersymmetric sigma model.
\newblock {\em Comm. Math. Phys.}, 300(3):659--671, 2010.

\bibitem[DSZ10]{disertori-spencer-zirnbauer2010}
M.~Disertori, T.~Spencer, and M.R. Zirnbauer.
\newblock Quasi-diffusion in a 3{D} supersymmetric hyperbolic sigma model.
\newblock {\em Comm. Math. Phys.}, 300(2):435--486, 2010.

\bibitem[Fre21]{fresta2021}
L.~Fresta.
\newblock Supersymmetric cluster expansions and applications to random
  {S}chr\"{o}dinger operators.
\newblock {\em Math. Phys. Anal. Geom.}, 24(1):Paper No. 4, 41, 2021.

\bibitem[KP21]{kozma-peled2021}
G.~Kozma and R.~Peled.
\newblock Power-law decay of weights and recurrence of the two-dimensional
  {VRJP}.
\newblock {\em Electron. J. Probab.}, 26:Paper No. 82, 19, 2021.

\bibitem[MRT19]{merkl-rolles-tarres2019}
F.~Merkl, S.W.W. Rolles, and P.~Tarr\`es.
\newblock Convergence of vertex-reinforced jump processes to an extension of
  the supersymmetric hyperbolic nonlinear sigma model.
\newblock {\em Probab. Theory Related Fields}, 173(3-4):1349--1387, 2019.

\bibitem[Sab21]{sabot-polynomial-decay2021}
C.~Sabot.
\newblock Polynomial localization of the 2{D}-vertex reinforced jump process.
\newblock {\em Electron. Commun. Probab.}, 26:Paper No. 1, 9, 2021.

\bibitem[ST15]{sabot-tarres2012}
C.~Sabot and P.~Tarr{\`e}s.
\newblock Edge-reinforced random walk, vertex-reinforced jump process and the
  supersymmetric hyperbolic sigma model.
\newblock {\em JEMS}, 17(9):2353--2378, 2015.

\bibitem[STZ17]{sabot-tarres-zeng2017}
C.~Sabot, P.~Tarr\`es, and X.~Zeng.
\newblock The vertex reinforced jump process and a random {S}chr\"odinger
  operator on finite graphs.
\newblock {\em Ann. Probab.}, 45(6A):3967--3986, 2017.

\bibitem[Swa20]{phdthesis-swan2020}
A.~Swan.
\newblock Superprobability on graphs ({PhD} thesis), 2020.

\bibitem[SZ19]{sabot-zeng15}
C.~Sabot and X.~Zeng.
\newblock A random {S}chr\"{o}dinger operator associated with the vertex
  reinforced jump process on infinite graphs.
\newblock {\em J. Amer. Math. Soc.}, 32(2):311--349, 2019.

\bibitem[Zir91]{zirnbauer-91}
M.R. Zirnbauer.
\newblock Fourier analysis on a hyperbolic supermanifold with constant
  curvature.
\newblock {\em Comm. Math. Phys.}, 141(3):503--522, 1991.

\end{thebibliography}
\end{document}